\documentclass[a4paper]{article}

\usepackage[latin1]{inputenc}
\usepackage[T1]{fontenc}

\usepackage[french,english]{babel}

\usepackage{amssymb,amsmath,amsthm,amscd,amsfonts}
\usepackage{mathrsfs}

\usepackage{subfig}
\usepackage{graphicx}
\usepackage{pstricks}

\usepackage{hyperref}
\usepackage{url}

\def\ie{\emph{i.e. }}

\def\R{\mathbb{R}}
\def\C{\mathbb{C}}
\def\N{\mathbb{N}}
\def\Z{\mathbb{Z}}
\def\P{\mathbb{P}^1}
\def\L{\mathbb{L}}
\def\H{\mathbb{H}}

\def\l{\lambda}
\def\a{\alpha}
\def\b{\beta}
\def\t{\theta}
\def\d{\delta}
\def\e{\varepsilon}
\def\ga{\gamma}

\def\X{\mathcal{X}}
\def\B{\mathcal{B}}
\def\D{\mathcal{D}}
\def\p{\mathcal{P}}
\def\E{\mathcal{E}}
\def\A{\mathcal{A}}

\def\w{\widetilde}

\def\g{\left(G,H \right) }

\def\gxt{\left(G(x,t),H(x,t) \right) }

\def\Y{\mathbf{Y}}

\def\ssm{\smallsetminus}

\newcommand{\dd}{\mathrm{d}}

\DeclareMathOperator{\tr}{Tr}

\DeclareMathOperator{\I}{I}
\DeclareMathOperator{\GL}{GL}
\DeclareMathOperator{\SL}{SL}
\DeclareMathOperator{\SO}{SO}
\DeclareMathOperator{\SU}{SU} 
\DeclareMathOperator{\OO}{O}
\DeclareMathOperator{\MM}{M}

\theoremstyle{plain}
\newtheorem{thm}{Theorem}%[section]

\newtheorem{prop}{Proposition}[section]
\newtheorem{lemma}{Lemma}[section]

\newtheorem*{thm*}{Theorem}

\theoremstyle{definition}
\newtheorem{defn}{Definition}[section]
\newtheorem*{nota}{Notation}
\newtheorem{rem}{Remark}[section]

\usepackage{enumitem}

\title{The Plateau problem for polygonal boundary curves in Minkowski $3$-space}
\author{Laura \textsc{Desideri}%
\thanks{Mathematisches Institut, Eberhard Karls Universit\"at T\"ubingen, Auf der Morgenstelle 10, 72076 T\"ubingen, Germany (email: \texttt{desideri@math.jussieu.fr})}}
\date{}

\begin{document}

\maketitle

\begin{abstract}
We apply Garnier's method to solve the Plateau problem for maximal surfaces in Minkowski $3$-space. Our study relies on the improved version we gave in~\cite{Desideri} of Garnier's resolution~\cite{Garnier28} of the Plateau problem for polygonal boundary curves in Euclidean $3$-space. Since in Minkowski space the method does not allow us to avoid the existence of singularities, the appropriate framework is to consider \emph{maxfaces} --- generalized maximal surfaces without branch points, introduced by Umehara and Yamada~\cite{UmeharaYamada06}. We prove that any given spacelike polygonal curve in generic position in Minkowski $3$-space bounds at least one maxface of disk-type. This is a new result, since the only known result for the Plateau problem in Minkowski space deals with boundary curves of regularity $C^{3,\a}$~\cite{Quien}.
\end{abstract}

%\begin{otherlanguage}{francais}
%
%\begin{abstract}
%On applique la m\'ethode de Garnier au probl\`eme de Plateau pour les surfaces maximales dans l'espace de Minkowski de dimension trois. Notre \'etude repose sur la version plus achev\'ee~\cite{Desideri} que l'on a propos\'ee de la r\'esolution obtenue par Garnier~\cite{Garnier28} du probl\`eme de Plateau pour des bords polygonaux dans l'espace euclidien. Dans l'espace de Minkowski, la m\'ethode de Garnier ne nous permet pas d'exclure l'existence de singularit\'es sur les surfaces que l'on construit : le cadre appropri\'e consiste \`a consid\'erer des \emph{maxfaces} --- des surfaces maximales g\'en\'erali\'ees sans point de branchement, introduites par Umehara et Yamada~\cite{UmeharaYamada06}. On montre que tout polygone de type espace en position g\'en\'erique est le bord d'au moins une maxface ayant la topologie du disque. Ceci constitue un r\'esultat nouveau, puisque le seul r\'esultat connu pour le probl\`eme de Plateau dans l'espace de Minkowski concerne des courbes de bord de r\'egularit\'e $C^{3,\a}$~\cite{Quien}.
%\end{abstract}
%
%\end{otherlanguage}

\bigskip


\bigskip


%\tableofcontents

\section*{Introduction}
\addcontentsline{toc}{section}{Introduction}

An immersion of a Riemann surface into the three-dimensional Minkowski space $\L^3=\left( \R^3, \dd X_1^2 + \dd X_2^2 - \dd X_3^2 \right)$ is said to be \emph{maximal} if it is spacelike, and if its mean curvature vanishes everywhere. Since it is a well-known fact that the only complete maximal surfaces are spacelike planes, maximal surfaces with singularities are an increasing object of interest. O. Kobayashi~\cite{Kobayashi} investigated conelike singular points on maximal surfaces. Many examples with conelike singular points were then found and studied by F. J. L\'opez, R. L\'opez, and R. Souam~\cite{LopezLopezSouam}, Fern\'andez and F. J. L\'opez~\cite{FernandezLopez}, and Fern\'andez, F. J. L\'opez and R.  Souam~\cite{FernandezLopezSouam05}, \cite{FernandezLopezSouam07}.

To study more general singularities, F. Estudillo and A. Romero~\cite{EstudilloRomero} defined a notion of \emph{generalized maximal surface}. These surfaces may have two types of singularities: branch points, and points at which the normal vector $N$ (which is well-defined and timelike at regular points) is still well-defined, but belongs to the light-cone. They also provided criteria for such surfaces to be planes. In~\cite{UmeharaYamada06}, M. Umehara and K. Yamada introduced \emph{maxfaces}, which are generalized maximal surfaces without branch point (see Section~\ref{section-max}). Maxfaces have analytical curves of singularities. Umehara and Yamada actually introduced maxfaces as projections into $\L^3$ of null holomorphic immersions into $\C^3$. They proved in~\cite{UmeharaYamada06} an Osserman type inequality for maxfaces. Umehara and Yamada~\cite{UmeharaYamada06}, and Fujimori, Saji, Umehara, and Yamada~\cite{UmeharaYamada08} studied singularities of maxfaces, which generically are cuspidal edges, swallowtails and cuspidal cross caps. Fujimori, Rossman, Umehara, Yamada, and Yang~\cite{RossmanUmeharaYamada} constructed new higher genus maxfaces.

%Recently, Imaizumi [15] studied the asymptotic behavior of maximal surfaces, and Imaizumi- Kato [16] gave a classification of maximal surfaces of genus zero with at most three embedded ends.

For the Dirichlet problem the existence of regular spacelike graphs of prescribed mean curvature was shown by R. Bartnik and L. Simon~\cite{BartnikSimon} in $\L^{n+1}$, and independently by C. Gerhardt~\cite{Gerhardt} in Lorentz manifolds with a product structure. A. A. Klyachin and V. M. Miklyukov~\cite{KlyachinMiklyukov} gave results on the existence of solutions, with a finite number of isolated singularities, to the maximal hypersurface equation in $\L^{n+1}$ with prescribed boundary conditions. Concerning the Plateau problem, N. Quien proved that the boundary of any $C^{3,\a}$ spacelike hypersurface in $\L^{n+1}$ also bounds a regular maximal hypersurface --- which is not always a graph. His assumption on the boundary is not superfluous, since there exist spacelike curves in $\L^3$ which bound no regular spacelike surface at all, and in particular, no maximal one. He also provided in $\L^3$ a sufficient condition on the boundary for uniqueness to the Plateau problem.

In 1928, R. Garnier published a resolution of the Plateau problem~\cite{Garnier28} for polygonal boundary curves in Euclidean three-space. His proof relies on a resolution of the Riemann--Hilbert problem, and on isomonodromic deformations of Fuchsian equations. His paper is in places really complicated, and even obscure or incomplete, which may explain why it seems to have been forgotten. In the present paper, leaning on the more accomplished version we gave~\cite{Desideri} of his ideas, we apply Garnier's method when the ambient space is Minkowski space $\L^3$. As in Euclidean space, Garnier's method enables us to avoid the existence of branch point, but not of the second (new) possible type of singularities: we shall construct \emph{maxfaces}. In one hand, considering we are not looking for complete maximal surfaces, we should not need to authorize singularities to solve the Plateau problem in $\L^3$. But since we consider \emph{all} spacelike polygonal curves in generic position, some of them do not bound any regular maximal surface at all. We obtain the following result.
%Precisely, for any given spacelike polygonal curve in generic position, Garnier's method enable us to construct all maxfaces spanning this curve. Let us notice that for some spacelike polygonal curve, none of these maxfaces is regular and for others, there exist regular ones. Unfortunately, the method does not allow us to distinguish these two cases.

\begin{thm}
Every possibly unclosed spacelike polygonal curve $P\subset\L^3$ in generic position bounds a generalized maximal disk, without branch point. Moreover, if $P$ is unclosed, this maximal disk has a helicoidal end.
\label{thm-plateau-max}
\end{thm}

Here we say that a possibly unclosed spacelike polygonal curve with $n+3$ edges is in generic position if the $(n+3)$-tuple of spacelike oriented directions of its edges $D=(D_1,\ldots,D_{n+3})$ belongs to the space $\D^n$ introduced in Section~\ref{section-max}, Definition~\ref{def-Dn}. It mainly sets that any two arbitrary directions are never parallel, and that the directions $D_{n+2}$ and $D_{n+3}$ are non-coplanar with any other direction $D_i$.

For every direction $D\in\D^n$ we introduce the space $\p^n_D$ of possibly unclosed polygonal curves with $n+3$ vertices and of oriented direction $D$, defined up to translations and homotheties of positive scale factor:  such polygonal curves are characterized by $n$ ratios of edge-lengths (between their $n+1$ finite lengths), and $\p^n_D$ is  thus isomorphic to $(0,+\infty)^n$. We also define the space $\X^n_D$ of maxfaces of disk-type with a polygonal boundary curve $P\in\p^n_D$, and with a helicoidal end if $P$ is unclosed, up to translations and homotheties of positive scale factor too. We can then paraphrase Theorem~\ref{thm-plateau-max}: it amounts  to proving that for any oriented direction $D\in\D^n$ the following map is surjective
\begin{align*}
\X^n_D &\longrightarrow \p^n_D\\
M & \longmapsto   \partial M.
\end{align*}
To this purpose we will fit Garnier's method to the case of maxfaces in Minkowski space: we establish an explicit one-to-one correspondence between an appropriate class of Fuchsian equations, denoted by $\E^n_D$, and the space $\X^n_D$, and we then prove that the following explicit composition is surjective
\[
\E^n_D \stackrel{1:1}{\longrightarrow}\X^n_D \longrightarrow \p^n_D \stackrel{\sim}{\longrightarrow}  (0,+\infty)^n.
\]

The space $\E^n_D$ is defined as follow. Since we consider maxfaces of disk-type, we can always assume, without loss of generality, that they are defined on the upper half-plane
\[
\C_+ = \{x\in \C \, , \; \Im (x)>0\}.
\]
Thanks to the spinor Weierstrass representation, such a maxface $X:\C_+\to\L^3$ is determined by two functions $G$ and $H$ holomorphic on the upper half-plane, without common zero and such that the moduli $|G|$ and $|H|$ do not coincide everywhere. The singularities of $X$ are the points where the equality $|G|=|H|$ holds. When the image of $X$ does not lie in a plane, the functions $G$ and $H$ are linearly independent, and are thus solutions of a unique second-order linear differential equation
\begin{equation}
y''+p(x)y'+q(x)y=0.
\tag{$E$}\label{E-intro}
\end{equation}
This is the equation associated with the maxface $X$. When the maxface $X$ represents a surface with a polygonal boundary curve, it appears that there is a nice correspondence between the geometry of $X$ and analytical properties of its associated equation~\eqref{E-intro}. We denote by $\E^n_D$ the space of all equations~\eqref{E-intro} defined by this way from a maxface $X\in\X^n_D$.

\bigskip

After generalities about maxfaces, and maxfaces with a polygonal boundary curve in Section~\ref{section-max}, and about Fuchsian equations in Section~\ref{section-equ-sys}, we provide in Section~\ref{section-equ-fu} a characterization of the space $\E^n_D$. The singularities of an  equation~\eqref{E-intro} in $\E^n_D$ are of two types: the pre-images of the vertices of the polygonal curve of its associated maxface $X$, which are real:
\[
 t_1<\cdots<t_n< t_{n+1}=0, \ t_{n+2}=1,\ t_{n+3}=\infty,
\]
and the umbilics, which are apparent singularities. By applying the Schwarz reflection principle, we prove that Equation~\eqref{E-intro} can be extended on the whole Riemann sphere $\P$, on which it is a real Fuchsian equation, and we determine how the Weierstrass data are transformed around the vertices $t_i$: we prove that the monodromy of Equation~\eqref{E-intro} is entirely determined by the oriented direction $D$ of the polygonal boundary curve.

Then, in Section~\ref{section-isomono}, for every direction $D\in\D^n$, we describe explicitly the isomonodromic space $\E^n_D$. Actually, it is more suitable to deal with Fuchsian systems instead of equations. By means of isomonodromic deformations, we obtain that $\E^n_D$ is parametrized by the position of the non-apparent singularities $t=\left(t_1, \ldots, t_n\right)$ varying in the simplex
\[
 \pi^n=\left\{ (t_1,\ldots,t_n)\in\R^n\ \big|\ t_1<\cdots<t_n<0\right\}.
\]
This also provides an explicit description of the space $\X^n_D=\left(X_D(t), t\in\pi^n\right)$, and of the family $\left(P_D(t), t\in\pi^n\right)\subset\p^n_D$ of polygonal curves of direction $D$ which bound at least one maxface of disk-type. The way these objects depend on $t$ is given by the Schlesinger system, a completely integrable system enjoying the Painlev\'e property.

Finally, to conclude, we study in Section~\ref{section-ratio} the $n$-tuple of length ratios of the polygons $P_D(t)$, denoted by $F_D(t)\in (0,+\infty)^n$. Solving the Plateau problem is equivalent to prove that the function $F_D:\pi^n\to (0,+\infty)^n$ is surjective. It constitutes the most difficult part of the proof of Theorem~\ref{thm-plateau-max}. It is based on the behavior of the solutions of the Schlesinger system at its fixed singularities, \ie at the boundary of the simplex $\pi^n$, which has been studied by Sato, Miwa and Jimbo~\cite{SMJ}. We then conclude by an induction on the the number $n+3$ of vertices, and by a degree argument.

\bigskip

We remind the main steps of the construction developed in~\cite{Desideri}, and we only provide the proofs which differ from the Euclidean case: they concern the expression of the monodromy, derived from the Schwarz reflection principle, and the expression of the length-ratio function $F_D$, which is more complicated because of the existence of singularities on maxfaces. We simplify the expression of $F_D$ in Minkowski space to prove that its behavior is the same as its analogue in Euclidean space. We then do not have to consider the technical details of the last part of the resolution, about its behavior at the boundary of $\pi^n$.

\bigskip

\noindent \textsc{Acknowledgments.} I would like to thank to my
advisor Fr\'ed\'eric H\'elein for helpful and interesting
discussions. I am also indebted to Rabah Souam for useful
comments.

%%%%%%%%%%%%%%%%%%%%%%%%%%%%%%%%%%%%%%%%%%%%

\section{Maximal surfaces and maxfaces}
\label{section-max}

\subsection{Maximal surfaces}

The three-dimensional Minkowski space $\L^3$ is the affine space $\R^3$ endowed with the Lorentzian metric
\[
 \left\langle \ , \, \right\rangle = \dd X_1^2 + \dd X_2^2 - \dd X_3^2.
\]
We say that a vector $V\in \R^3\ssm\{0\}$ is spacelike, timelike or lightlike if $\left\langle V,V \right\rangle$ is respectively positive, negative or zero. The vector $0$ is spacelike by definition. We denote by
\[
 \mathbb{H}^2=\left\{ (X_1,X_2,X_3)\in\R^3\ \big\vert\ X_1^2+X_2^2-X_3^2=-1 \right\}
\]
the hyperbolic sphere in $\L^3$ of constant intrinsic curvature $-1$. The sphere $\mathbb H^2$ has two connected components $\mathbb H^2_+$ and $\mathbb H^2_-$, characterized by the sign of $X_3$. The stereographic projection $\pi$ of $\mathbb H^2$ is defined by
\[
 \pi:\overline\C\ssm\{|x|=1\} \to \mathbb H^2, \qquad \pi(x) = \left(\frac{2\Im x}{|x|^2-1},\frac{2\Re x}{|x|^2-1},\frac{|x|^2+1}{|x|^2-1}\right)
\]
and $\pi(\infty)=(0,0,1)$.

An immersion $X:\Sigma\to\L^3$ of a $2$-manifold $\Sigma$ into $\L^3$ is said to be \emph{spacelike} if its induced metric is Riemannian. The Gauss map $N:\Sigma\to\mathbb H^2$ of the immersion $X$ is then globally well-defined, with values into one of the connected components of $\mathbb H^2$. We can thus regard $\Sigma$ as a Riemann surface and $X$ as a conformal immersion.

A spacelike immersion $X:\Sigma\to\L^3$ is said to be \emph{maximal} if its mean curvature vanishes identically. Since we are currently interested in disk-type surfaces only, for simplicity we give the spinor Weierstrass representation only for maximal immersions defined on $\C_+$. See~\cite{KuSch} for a more general setting, and for more details on the spinor representation in Euclidean case. We denote by $x$ the standard coordinate on $\C_+$.

\begin{thm}
Let $x_0$ be a point of the upper half-plane  $\C_+$.

For any maximal conformal immersion $X:\C_+\to\L^3$ there exist a point $X_0\in\R^3$, and two holomorphic functions $G$ and $H$ on $\C_+$ such that $|G(x)|\neq|H(x)|$ for all $x\in\C_+$ satisfying
\begin{equation}
 X(x) = X_0 + \Re \int_{x_0}^x
\Big(
 H^2 - G^2,
 i\left(G^2 + H^2\right),
 2iGH
\Big)
\dd x.
\label{rep-W-regular}
\end{equation}

Conversely, any holomorphic functions $G$ and $H$ on $\C_+$ such that $|G(x)|\neq|H(x)|$ for all $x\in\C_+$ define by~\eqref{rep-W-regular} a maximal conformal immersion $X:\C_+\to\L^{3}$.
\label{thm-W1}
\end{thm}

We call the couple $\g$ the (spinor) Weierstrass data of the immersion $X$. The stereographic projection of the Gauss map of $X$ is then given by
\[
 g = \pi\circ N = -\frac GH.
\]
The induced metric and the Hopf differential are expressed in terms of the Weierstrass data by
\begin{equation}
 \dd s^2=\left(|G|^2-|H|^2 \right)^2|\dd x|^2 , \qquad Q=i\left( GH'-HG'\right) \dd x^2
\label{hopf-metric}
\end{equation}
where $\, '$ denotes the differentiation with respect to $x$.

\subsection{Maxfaces}

Unfortunatly, we can not control the condition
\[
 \forall x\in\C_+ \quad |G(x)|\neq|H(x)|
\]
of Theorem~\ref{thm-W1} by the equation associated by Garnier's method with the maximal conformal immersion $X$ of Weierstrass data $\g$. Following Garnier's method to solve the Plateau problem, we are thus led to construct maximal surfaces with singularities. We can be more precise about these singularities.

In~\cite{EstudilloRomero}, F. Estudillo and A. Romero defined a notion of \emph{generalized maximal surface} as follow. Let $X:\Sigma\to\L^3$ be a differentiable map on a Riemann surface $\Sigma$. The map $X$ is then called a generalized maximal surface if we have
\begin{itemize}
 \item $\Phi:=\partial X/\partial x : \Sigma\to\C^3$ is holomorphic ;
 \item $\left(\Phi_1\right)^2 + \left(\Phi_2\right)^2 - \left(\Phi_3\right)^2 = 0$ ;
 \item $\left|\Phi_1\right|^2 + \left|\Phi_2\right|^2 - \left|\Phi_3\right|^2$ is not identically zero.
\end{itemize}
The singular points are the points where $\left|\Phi_1\right|^2 + \left|\Phi_2\right|^2 - \left|\Phi_3\right|^2 = 0$ holds. They are of two types: the set $B$ of the isolated zeros of the holomorphic function $\Phi$, and the set $A$ of the points where $|g|=1$. In general $A\cap B\neq\emptyset$. Points in $B\ssm (A\cap B)$ are isolated, whereas points in $A$ are not.

In terms of the spinor Weierstrass representation, a map on $\C_+$ satisfying the first two previous conditions can always be written in the form~\eqref{rep-W-regular} where the functions $G$ and $H$ are holomorphic. Then we have $\left|\Phi_1\right|^2 + \left|\Phi_2\right|^2 - \left|\Phi_3\right|^2 = |G|^2-|H|^2 $: singular points in $B$ are the common zeros of $G$ and $H$, and points in $A$ are those where we have $ \left|G / H\right| =1$.

In~\cite{UmeharaYamada06}, M. Umehara and K. Yamada introduced \emph{maxfaces}, which are generalized maximal surfaces in the sense of Estudillo and Romero satisfying $B=\emptyset$. The terminology is not so clear. Estudillo and Romero call branch point any singular point of a generalized maximal surface. But following Umehara and Yamada, we will call branch points the points belonging to the set $B$: maxfaces are in this sense generalized maximal surfaces without branch point. Umehara and Yamada actually introduced maxfaces as projections into $\L^3$ of null holomorphic immersions into $\C^3$. They proved in~\cite{UmeharaYamada06} an Osserman type inequality for maxfaces. Umehara and Yamada~\cite{UmeharaYamada06}, Fujimori, Saji, Umehara, and Yamada~\cite{UmeharaYamada08} and Fujimori, Rossman, Umehara, Yamada, and Yang~\cite{RossmanUmeharaYamada} studied singularities of maxfaces, and constructed new examples.

We have the following spinor Weierstrass representation for maxfaces of disk-type.

\begin{thm}[The spinor Weierstrass representation for maxfaces]
Let $x_0$ be a point of the upper half-plane  $\C_+$.

For any maxface $X:\C_+\to\L^3$ there exist a point $X_0\in\R^3$, and two holomorphic functions $G$ and $H$ on $\C_+$ without common zero and such that $|G|-|H|$ is not identically zero on $\C_+$ satisfying
\begin{equation}
 X(x) = X_0 + \Re \int_{x_0}^x
\Big(
 H^2 - G^2,
 i\left(G^2 + H^2\right),
 2iGH
\Big)
\dd x.
\label{rep-W}
\end{equation}

Conversely, any holomorphic functions $G$ and $H$ on $\C_+$ without common zero and such that $|G|-|H|$ is not identically zero on $\C_+$ define by~\eqref{rep-W-regular} a maxface $X:\C_+\to\L^{3}$.
\label{thm-W2}
\end{thm}

The induced metric and the Hopf differential of a maxface are still expressed in terms of its Weierstrass data by~\eqref{hopf-metric}.

It is sometimes convenient to describe the space $\L^3$ in terms of $2\times2$ matrices, by identifying each vector $X=(X_1,X_2,X_3)^t\in\L^3$ with the Hermitian matrix $\w X\in\MM(2,\C)$ defined by
\begin{equation}
 \w X =
 \begin{pmatrix}
  X_3      & X_2+iX_1\\
  X_2-iX_1 & X_3
 \end{pmatrix}.
\label{(X-tilde)}
\end{equation}
This induces an isomorphism from $\L^3$ into the space $L^{2,1}$ of Hermitian matrices of the form~\eqref{(X-tilde)}. The metric of $\L^3$ is then given by 
\[
\left\langle X, Y\right\rangle = \frac12 \tr \w X \w Y, \quad \text{and} \quad \left\langle X, X\right\rangle = -\det \w X.
\]
The Lorentz group $\OO(2,1)$ is the group of linear isometries of $\L^3$: it is constituted of the matrices $M\in \MM(3,\R)$ satisfying
\[
 MQM^t = Q
\qquad \text{where }
Q = \text{diag} (1,1,-1) .
\]
The group $\OO(2,1)$ has four connected components. The restricted Lorentz group $\SO^+(2,1)$ is the connected component of the identity, it is formed of all isometries preserving orientations of both space and time. The eigenvalues of a matrix $R\in\SO^+(2,1)$ are of the form $(1,e^{\varphi},e^{-\varphi})$ or $(1,e^{i\varphi},e^{-i\varphi})$, with $\varphi\in\R$, according to whether the axis of $R$ is spacelike or timelike.

Let us recall that the group $\SU(1,1)$ is the group of matrices $A\in \MM(2,\C)$ of determinant equal to $1$ verifying
\[
 \overline A^t
 \begin{pmatrix}
  1 & 0\\
  0 & -1
 \end{pmatrix}
 A =
 \begin{pmatrix}
  1 & 0\\
  0 & -1
 \end{pmatrix} .
\]
It is isomorphic to $\SL(2,\R)$, and it is described as follow
\[
 \SU(1,1) = \left\{\left.
\begin{pmatrix}
 a & b\\
 \overline b & \overline a
\end{pmatrix}
\ \right\vert \ a,b\in \C,\ \ a\overline a - b\overline b = 1
\right\}.
\]
For every matrix $A\in \SU(1,1)$, the following map is well-defined
\begin{align*}
 R_A :\ L^{2,1} &\to L^{2,1}\\
     M  &\mapsto \overline A^t M A
\end{align*}
and it is a linear isometry of $L^{2,1}$. By identifying $\SO^+(L^{2,1})$ and $\SO^+(2,1)$, we obtain the homomorphism
\begin{align}
 R :\ \SU(1,1) &\to \SO^+(2,1)\label{spin-covering}\\
     A  &\mapsto R_A\nonumber
\end{align}
which is the double universal cover of $\SO^+(2,1)$ by the group $Spin(2,1)\simeq \SU(1,1)$. We can write it as follow. If the direct isometry $R\in\SO^+(2,1)$ has a timelike unitary axis $\d=(\d_1,\d_2,\d_3)$, and an ``angle'' $\varphi$, \ie has $e^{\pm i\varphi}$ for eigenvalues, its pre-images by~\eqref{spin-covering} are $A$ and $-A$, where
\[
A = \cos\left( \frac\varphi2 \right) \I_2 + \sin \left( \frac\varphi2 \right)
\begin{pmatrix}
i\d_3 & \d_1+i\d_2\\
\d_1-i\d_2 & -i\d_3
\end{pmatrix}.
\]
If $R\in\SO^+(2,1)$ has a spacelike unitary axis $\d=(\d_1,\d_2,\d_3)$, and $e^{\pm \varphi}$ for eigenvalues, its pre-images by~\eqref{spin-covering} are $A$ and $-A$, where
\[
A = \cosh\left( \frac\varphi2 \right) \I_2 + \sinh \left( \frac\varphi2 \right)
\begin{pmatrix}
i\d_3 & \d_1+i\d_2\\
\d_1-i\d_2 & -i\d_3
\end{pmatrix}.
\]
The following proposition is the analogue of the well-known property of spinor representations into Euclidean space.

\begin{prop}
Let $X:\Sigma\to\L^3$ be a maxface of Weierstrass data $\g$, and $A$ a matrix in $\SU(1,1)$. Then the vector $\g A$ constitutes the Weierstrass data of the maxface $ R_A\left( X\right) $ which is the image of $X$ by the isometry $R_A\in \SO^+(2,1)$.
\label{prop-spin}
\end{prop}

\begin{proof}
We use the previous description of $\L^3$. A direct computation shows that
 \[
 \w X(x) = i\int_{x_0}^x
 \begin{pmatrix}
  GH & H^2\\
  G^2 & GH
 \end{pmatrix}
 \dd x
 -i\int_{x_0}^x
 \begin{pmatrix}
  \overline{GH}  & \overline G^2\\
  \overline H^2  & \overline{GH}
 \end{pmatrix}
 \dd\bar x,
\]
which can be written as
\[
 \w X(x) = \int_{x_0}^x J\cdot Y^t\cdot Y \dd x
 - \int_{x_0}^x
 \overline Y^t\cdot \overline Y\cdot J
 \dd\bar x,
\]
where we have set $Y=\g$, and
$
 J =
\begin{pmatrix}
 0 & i\\
 i & 0
\end{pmatrix}.
$
From the identity
\begin{equation}
 \forall A\in \SU(1,1) \qquad J A= \overline A J,
\label{SU}
\end{equation}
we finally get
\[
 \overline A^t \w X(x) A =
 \int_{x_0}^x
 J\cdot \left(YA\right)^t \cdot \left(YA\right)
 \dd x
 -\int_{x_0}^x
 \overline{\left(YA\right)}^t\cdot \overline{\left(YA\right)}\cdot J
 \dd\bar x.
\]
\end{proof}

\subsection{Maxfaces with a polygonal boundary curve}

We introduce the appropriate spaces and notations for the maxfaces we intend to construct, and for their polygonal boundary curves. First of all,  there are some natural assumptions we should state on the polygonal boundary curves we consider, and some others less natural that we will need during the proof.

Let us consider a polygonal curve $P\subset\R^3$ with $n+3$ vertices $a_1,\ldots, a_{n+3}$. We denote by $D_i$ the oriented direction of the straight line $(a_i,a_{i+1})$, and by $u_i$ a direction vector inducing the orientation of $D_i$. Moreover, we authorize the polygon $P$ to be \emph{possibly unclosed}, that is to say  the two half-lines derived from $a_1$ and $a_{n+2}$, and of respective oriented directions $-D_{n+3}$ and $D_{n+2}$ do not necessary intersect each other. When the curve is unclosed, we then say that  the vertex $a_{n+3}$ is at infinity.

We say that the polygonal curve $P$ is \emph{non-degenerate} if the cross products $u_{i-1}\times u_i$ are all non-zero ($i=1,\ldots,n+3$, the indices are always defined modulo $n+3$). We can then define at each vertex $a_i$:
\begin{itemize}
 \item the measure $\t_i\pi$ of the angle between $u_i$ and $u_{i-1}$ (the exterior angle of $P$) such that $0<\t_i<1$,
 \item a normal vector to the polygon $P$: $v_i = -u_{i-1}\times u_i$.
\end{itemize}
We say that a non-degenerate polygonal curve $P$ is \emph{spacelike} if
\begin{itemize}
 \item the direction vectors are spacelike: $\left\langle u_i ,u_i \right\rangle>0$ ($i=1,\ldots,n+3$), and if
 \item the normal vectors are timelike: $\left\langle v_i ,v_i \right\rangle<0$ ($i=1,\ldots,n+3$).
\end{itemize}
When the curve $P$ is spacelike, we always assume that for every $i=1,\ldots,n+3$
\[
\left\langle u_i ,u_i \right\rangle=1, \quad \text{and} \quad \left\langle v_i ,v_i \right\rangle=-1.
\]

The results of Sections~\ref{section-equ-fu} and~\ref{section-isomono} concern all non-planar, non-degenerate spacelike polygonal curves. But to end the resolution of the Plateau problem we will need at Section~\ref{section-ratio} stronger assumptions on the polygonal boundary curves. Since we use an induction on the number $n+3$ of vertices, we have to introduce a family of polygons such that the conditions on the edge directions are passed on to any subsets of directions.

\begin{defn}
We define the set $\D^n$ of $(n+3)$-tuples $D=(D_1,\ldots,D_{n+3})$ of spacelike oriented directions in $\L^3$ satisfying the following properties
\begin{itemize}
 \item any two directions $D_i$ and $D_j$ ($i\neq j$) are non-collinear, and their common normal direction is timelike,
 \item for any $i\neq n+1,\ n+2$, the directions $D_i$, $D_{n+1}$ and $D_{n+2}$ are non-coplanar.
\end{itemize}
\label{def-Dn}
\end{defn}
%The space $\D^n$ is dense in the space of $(n+3)$-tuples of spacelike directions.

For every $(n+3)$-gon, we call the $(n+3)$-tuple $D$ of oriented directions of its edges its \emph{oriented direction}. We introduce the following spaces of polygonal curves.

\begin{nota}
For every $D\in\D^n$ we denote by $\p^n_D$ the quotient of the space of possibly unclosed $(n+3)$-gons in $\R^3$ of oriented direction $D$ by the group of translations and homotheties of positive scale factor:
\[
 \p^n_D := \left\{ (n+3)\text{-gons of oriented direction } D\right\}  \Big/_{\displaystyle \R^3\times\R^*_+ } .
\]
\end{nota}

Of course the space $\p^n_D$ contains all \emph{closed} polygonal curves of oriented direction $D$, when there exist any. Since there is no closure condition, a possibly unclosed $(n+3)$-gon with known oriented direction $D$ is characterized by the position of one vertex and the values of its first $n+1$ edge lengths --- which are always finite. A coordinate system on the space $\p^n_D$ is thus given by any choice of $n$ edge-length ratios between these $n+1$ lengths. We choose the following one:
\begin{equation}
 (r_1,\ldots,r_n):\p^n_D\to (0,+\infty)\,^n, \qquad r_i(P)=\frac{||a_ia_{i+1}||}{||a_{n+1}a_{n+2}||},
\label{coord-poly}
\end{equation}
where $a_1,\ldots,a_{n+3}$ are the vertices in $\R^3\cup\{\infty\}$ of any representative of $P\in\p^n_D$. We get the isomorphism
\[
 \p_D^n \simeq (0,+\infty)\,^n.
\]

The maxfaces we will construct in next sections via Garnier's method lie in the following spaces. Since we consider maxfaces of disk-type, we always assume without loss of generality that they are defined on the upper half-plane $\C_+$.

\begin{nota}
For every $D\in\D^n$ we denote by $\X^n_D$ the quotient by the group of translations and homotheties of positive scale factor of the space of maxfaces of disk-type $X:\C_+\to \L^3$ such that
\begin{itemize}
\item $X$ continuously extends onto $\overline\R=\R\cup\{\infty\}$, on which it homeomorphically parametrizes a polygon $P$ such that  $[P]\in\p^n_D$, and $X$ has no boundary branch point, except possibly at the vertices of $P$,
\item $X$ has a helicoidal end if $P$ is unclosed, and
\item $X$ is locally embedded around the vertices.
\end{itemize}
%which continously extend onto $\overline\R=\R\cup\{\infty\}$ such that $X\mid_{\overline\R}$ homeomorphically parametrizes a polygonal curve $P$ such that  $[P]\in\p^n_D$, such that $X$ has a helicoidal end if $P$ is unclosed, and is locally embedded around the vertices
%\[ \X^n_D := \left\{ \text{maxfaces } X \ | \ \ \partial\left( X( \C_+) \right) \in \p^n_D\right\} / \sim.\]
\end{nota}

Let us consider a maxface $X:\C_+\to \L^3$ such that $[X]\in\X^n_D$.  We denote by $P$ its polygonal boundary curve, and by $Y_0:=(G,H):\C_+\to \C^2$ its Weierstrass data, which are holomorphic on the upper half-plane $\C_+$. We denote by
\[
    t_1<\cdots<t_{n+3}
\]
the points in $\overline\R$ which are the pre-images by $X$ of the vertices of the polygonal curve $P$. By composing the map $X$ with a homography, we can assume that
\[
    t_{n+1}=0, \quad t_{n+2}=1, \quad t_{n+3}=\infty,
\]
which entirely determines $X$. The function $Y_0(x)$ is then unique up to the sign. From the first assumption on $X$, we know that the function $Y_0(x)$ is continuous on the intervals $(t_i,t_{i+1})$ ($i=1,\ldots,n+3$). This is a natural assumption if we want  $X$ to extend trough its edges, thanks to the Schwarz reflection principle. Under this assumption, the Gauss map $N(x)$ of the maxface $X$ admits a limit at each vertex of $P$. We denote by $N(t_i)$ the limit Gauss vector at $x=t_i$, which lies in the hyperbolic sphere $\mathbb H^2$ and satisfies $N(t_i)=\pm v_i$. We thus know that the maxface $X$ has no singularity on a neighborhood of the vertices.

Following the previous notations on the polygonal boundary curve $P$, the third assumption on $X$ means that its angle at $a_i$ is $(1-\t_i)\pi\in(0,\pi)$ or $(1+\t_i)\pi\in(\pi,2\pi)$. As we will see in Section~\ref{section-equ-fu}, this infers that the branch points of $X$, which are necessarily located at the vertices, are of order $1$, and occur if and only if the angle of $X$ is $(1+\t_i)\pi$.

%%%%%%%%%%%%%%%%%%%%%%%%%%%%%%%%%%%%%%%%%%%%%%%%

\section{Fuchsian equations}
\label{section-equ-sys}

To prove Theorem~\ref{thm-plateau-max}, we intend to fit to the case of maxfaces the original Garnier's method to solve the Plateau problem in the three-dimensional Euclidean space. His approach relies on an explicit one-to-one correspondence between a class of Fuchsian equations and the space of minimal disks with a polygonal boundary curve. The idea is instead of looking for a minimal disk with a given boundary, to look rather for its associated equation. The main part of the resolution thus belongs to the domain of Fuchsian equations and Fuchsian systems. We present here generalities on this domain. We refer the reader to~\cite{IKSY} for a more complete description of the subject, especially for the definition of the Garnier system, and its relations with the Schlesinger system.

\paragraph{Local behavior}

Let us consider a second-order linear ordinary differential equation on the Riemann sphere $\P=\C\cup\{\infty\}$
\begin{equation}
y''+p(x)y'+q(x)y=0
\tag{$E_0$}\label{E0}
\end{equation}
where the coefficients $p(x)$ and $q(x)$ are meromorphic on $\P$ ($\ '$ denotes the differentiation with respect to the complex variable $x$). Let $S$ denote the singular set of Equation~\eqref{E0} (the poles of $p(x)$ and $q(x)$), which is finite. A singular point $x_0\in S$, $x_0\neq\infty$, is said \emph{Fuchsian} if it is a simple pole for $p(x)$, and a simple or double pole for $q(x)$.

For example, the point $x=0$ is a Fuchsian singularity if and only if we have $p(x)=\w p_0(x)/x$ and $q(x)=\w q_0(x)/x^2$ where the functions $\w p_0(x)$ and $\w q_0(x)$ are holomorphic at $x=0$. We can then define the \emph{exponents} at $x=0$, which are the complex numbers $\a$ and $\b$ solving the following quadratic equation:
\[
 s^2+(p_0-1)s+q_0=0,
\]
where $p_0=\w p_0(0)$ and $q_0=\w q_0(0)$. The singular point $x=0$ is \emph{non-logarithmic} if there is a fundamental system of solutions of the form:
\[
    g(x)=x^\a\w g(x), \quad
    h(x)=x^\b\w h(x),
\]
where $\w g$ and $\w h$ are holomorphic, non-vanishing functions at $x=0$. This is always the case if the exponents $\a$ and $\b$ do not differ by an integer (the expression of solutions at a logarithmic singularity is a bit more complicated). The same results hold for any finite singularity\footnote{%
%%%%%%%%%%%%%%%%%%%%%%%%%%%%%%%%%%%%%%%%%%%%%%%%%%%%
Fuchsian singularities are always \emph{regular}, which means that the solutions of~\eqref{E0} have a polynomial growth at a Fuchsian singularity. This should be set more carefully since the solutions of~\eqref{E0} are multi-valued around a singularity (one has to restrict them onto sectors centered at the singularity on the universal covering space of $\P\ssm S$). The converse is also true, but for scalar equations only (and not for systems of equations).
}. %%%%%%%%%%%%%%%%%%%%%%%%%%%%%%%%%%%%%%%%%%%%%%%%%%

We determine the nature of the point $x=\infty$ by changing the independent variable $x$ into $w=1/x$ in Equation~\eqref{E0}, and by studying the new equation at the point $w=0$. We obtain by this way that $x=\infty$ is a Fuchsian singularity if and only if we have $p(1/w)=w\w p_\infty(w)$ and $q(1/w)=w^2\w q_\infty(w)$ where the functions $\w p_\infty(w)$ and $\w q_\infty(w)$ are holomorphic at $w=0$. If so, the characteristic equation at
infinity is:
\[
 s^2+(1-p_\infty)s+q_\infty=0,
\]
where $p_\infty=\w p_\infty(0)$ and $q_\infty=\w q_\infty(0)$.

Equation~\eqref{E0} is said \emph{Fuchsian} if all its singularities, including $x=\infty$, are Fuchsian. If Equation~\eqref{E0} is Fuchsian, with $r$ singular points $x_1, \ldots, x_r=\infty$ of respective exponents $\a_i$ and $\b_i$, then we obtain from the previous local description the following expression of the coefficients $p(x)$ and $q(x)$
\[
 p(x) = \sum_{i=1}^{r-1} \frac{p_i}{x-x_i}, \qquad
 q(x) = \sum_{i=1}^{r-1} \frac{q_i}{(x-x_i)^2} + \sum_{i=1}^{r-1} \frac{q'_i}{x-x_i},
\]
where $p_i=1-\a_i-\b_i$, $q_i=\a_i\b_i$ and $\sum q'_i =0$. The exponents of a Fuchsian equation are related together: the residue theorem applied to $p(x)$ gives us that $p_\infty = \sum_{i=1}^{r-1}p_i$. From the characteristic equations, we then infer \emph{the Fuchs relation}:
\begin{equation}
    \sum_{i=1}^r(\a_i+\b_i) = r-2.
\label{Fuchs}
\end{equation}

\paragraph{The Riemann--Hilbert problem}

Solutions of Equation~\eqref{E0} are holomorphic functions on the universal covering space of $\P\ssm S$. The \emph{monodromy} of Equation~\eqref{E0} is an equivalent class of representations of the fundamental group of $\P\ssm S$:
\[
 \rho : \pi_1\left(\P\ssm S, * \right) \to \GL(2,\C)
\]
that measures the lack of uniformity of solutions around the singularities. The \emph{Riemann--Hilbert problem} is to prove that there always exists a Fuchsian equation with a given monodromy and a given singular set. If the given monodromy is irreducible, then we get a positive answer, provided that we authorize additional parameters: the \emph{apparent singularities}. The apparent singularities are the singularities at which every solution is uniform. They are exactly the non-logarithmic singularities whose exponents are integers. If the singular set $S$ contains $r\geq3$ points, then the sufficient number of apparent singularities to obtain a positive answer to the Riemann--Hilbert problem is $r-3$ (see~\cite{Oht}).

\paragraph{Isomonodromic deformations}

If we suppose that Equation~\eqref{E0} depends on a variable parameter, how shall we describe the set of Fuchsian equations with a given common monodromy? For Fuchsian equations without logarithmic singularity, isomonodromic deformations are described by the Garnier system, a completely integrable Hamiltonian system generalizing the sixth Painlev\'e equation. But the Garnier system does not have the Painlev\'e property.

\begin{defn}
A differential equation
\[
 F\left(t,y,\frac{dy}{dt}, \ldots,\frac{d^py}{dt^p}\right) = 0,
\]
where the function $F\left(t,y_0,y_1, \ldots,y_p\right)$ is polynomial in $(y_0,y_1, \ldots,y_p)$ with meromorphic coefficients in $t$, has the \emph{Painlev\'e property} if it is free of movable branch point and movable essential singularity, \ie if the position of branch points and of essential singularities of its solutions does not depend on integration constants.
\label{def-prop-P}
\end{defn}

It is the main reason why, unlike Garnier, we mostly exclusively work with first-order $2\times2$ Fuchsian systems (instead of equations). When the position of singularities is varying, isomonodromic deformations of Fuchsian systems are described by the Schlesinger system~\eqref{schlesinger}, which is an integrable system, enjoying the Painlev\'e property.

%%%%%%%%%%%%%%%%%%%%%%%%%%%%%%%%%%%%%%%%%%%%%%%%

\section{Equations associated with maxfaces with a polygonal boundary curve}
\label{section-equ-fu}

Let us fix an oriented direction $D\in\D^n$, and consider a maxface $X:\C_+\to \L^3$ in $\X^n_D$. We use the notations introduced in Section~\ref{section-max}: we denote by $P\in\p^n_D$ the polygonal boundary curve of $X$, and by
\[
 Y_0:=(G,H),\quad  Y_0:\C_+\to \C^2
\]
its Weierstrass data, which are holomorphic on the upper half-plane $\C_+$. Since the image of $X$ does not lie in a plane, the Weierstrass data $G$ and $H$ are linearly independent. The function $Y_0$ is thus a fundamental system of solutions of a unique second-order linear ordinary differential equation
\begin{equation}
y''+p(x)y'+q(x)y=0.
\tag{$E$}\label{E}
\end{equation}
This equation is defined on the upper half-plane $\C_+$, its coefficients are expressed by the Weierstrass data by
\[
    p(x) = -\frac{GH''-HG''}{GH'-HG'}, \qquad
    q(x) = \frac{G'H''-H'G''}{GH'-HG'}.
\]
Notice that the Hopf differential of the maxface $X$, which is given by the Wronskian of $G$ and $H$ by $Q=i\left( GH'-HG'\right) \dd x^2$, satisfies $Q=i\exp \left( -\int p\right)\dd x^2$. We already see that the functions $p(x)$ and $q(x)$, which are meromorphic on $\C_+$, have two types of singularities:
\begin{itemize}
\item the pre-images $t_1<\cdots<t_n< t_{n+1}=0, \ t_{n+2}=1,\ t_{n+3}=\infty$ of the vertices of the polygon $P$, at which $Y_0(x)$ is singular,
\item the umbilics of the maxface $X$, \ie the zeros of its Hopf differential, at which $Y_0(x)$, and thus also every solution of Equation~\eqref{E}, is regular.
\end{itemize}
The umbilics are Fuchsian apparent singularities. We will prove that the $t_i$'s are Fuchsian singularities too. On the other hand, the singular points of the maxface $X$, \ie points where $|G|=|H|$ holds, are ordinary points of Equation~\eqref{E}.

Different maxfaces can define the same equation. For example, from Proposition~\ref{prop-spin}, we see that applying an isometry on $X$ keeps Equation~\eqref{E} unchanged, as well as applying a homothety. That is why we consider that Equation~\eqref{E} is defined by an element of $\X^n_D$, \ie by a maxface defined up to translations and homotheties of positive scale factor. An associated family of maxfaces corresponds also to the same equation. 

%Equation~\eqref{E} can be defined from every non-planar maxface. When the maxface $X$ has a polygonal boundary curve, it appears that there is a nice correspondence between the geometry of $X$ and analytical properties of its associated equation~\eqref{E}.
The aim of this section is to characterize second-order linear differential equations that come from a maxface with a polygonal boundary curve. Essentially these equations are real Fuchsian equations on the Riemann sphere $\P$, and their monodromy is determined by the oriented direction of the polygonal boundary curve.

\subsection{Monodromy and reality properties}

We denote by $S(t)$ the singular set of Equation~\eqref{E}
\[
 S(t):=\{t_1,\ldots, t_{n+3}\} \subset \overline\R.
\]
By geometrical considerations, we will determine the monodromy of the Weierstrass data $G$ and $H$ at these singular points. We will see that the monodromy of Equation~\eqref{E} is closely related to its reality properties, since they both express the Schwarz reflection principle. The following lemma enables us to extend Equation~\eqref{E} on the whole Riemann sphere.

\begin{lemma}
The coefficients $p(x)$ and $q(x)$ of Equation~\eqref{E} are real valued on $\overline\R\ssm S(t)$ and can be meromorphically extended onto $\P\ssm S(t)$. \label{lemma-p-q-real}
\end{lemma}

\begin{proof}
Lemma~\ref{lemma-p-q-real} is a direct consequence of the fact
that the maxface $X$ is bounded by pieces of spacelike straight
lines. One can easily deduce from the Weierstrass
representation~\eqref{rep-W} that the image by $X$ of the interval
$\left(t_i,t_{i+1} \right)$ is a piece of straight line directed
by the first-coordinate vector $e_1=(1,0,0)$ if and only if the
functions $G$ and $H$ are both real or purely imaginary on
$\left(t_i,t_{i+1} \right)$. By Proposition~\ref{prop-spin}, we
then know that there is a matrix $S_i\in\SU(1,1)$ such that the
function $\g\cdot S_i$ is real or purely imaginary on
$\left(t_i,t_{i+1} \right)$. By changing $S_i$ into $S_i\cdot
\text{diag}(i,-i)$, we can even assume that it is real. The matrix
$S_i$ is a pre-image by the universal cover~\eqref{spin-covering}
of a direct isometry of $\L^3$ mapping the direction vector $u_i$
into $e_1$ or $-e_1$. The coefficients $p(x)$ and $q(x)$ are thus real
valued on $\left(t_i,t_{i+1} \right)$.

We can thus extend the functions $p(x)$ and $q(x)$ uniformly onto the lower half-plane $\C_-=\{x\in \C \ | \ \Im (x)<0\}$ by setting for every $x\in\C_-$
\[
 p(x) := \overline{p(\bar{x})},\qquad
 q(x) := \overline{q(\bar{x})},
\]
and they are then meromorphic on $\P\ssm S(t)$.
\end{proof}

We introduce the following skew-linear map $\tau$ defined on the sheaf of holomorphic functions $\mathcal O_{\P}$ by
\begin{equation}
\begin{split}
 \tau :\
  \mathcal O_{\P}(\Omega)  & \longrightarrow \mathcal O_{\P}\left( \overline\Omega \right) \\
  f                    & \longmapsto \tau(f)= (x\mapsto \overline{f(\bar{x})}).
\end{split}
\label{def-tau}
\end{equation}
If the open set $\Omega\subset\P$ is connected and symmetric with respect to the real axis: $\overline\Omega=\Omega$, then a function $f\in\mathcal O_{\P}(\Omega)$ is real analytic if and only if $\tau(f)=f$.

The holomorphic function $\tau(Y_0)=\left(\tau(G),\tau(H)\right):\C_- \to\C^2$ is also the Weierstrass data of a maxface $X^-:\C_-\to\R^3$. A direct computation shows that the surface $X^-(\C_-)$ is symmetric to $X(\C_+)$ with respect to the first-coordinate axis $(O,e_1)$. This symmetry is an isometry of $\L^3$ which does not preserve the orientation of time: it belongs to $\SO^-(2,1)$. The surface $X^-(\C_-)$ can also be represented on the upper half-plane $\C_+$ by the maxface $X^+:\C_+\to\R^3$ of Weierstrass data $\left(G^+,H^+\right):\C_+ \to\C^2$ defined by
\[
    \begin{pmatrix}
        G^+ & H^+
    \end{pmatrix}
    =
    \begin{pmatrix}
        G & H
    \end{pmatrix}
    J,
\qquad
\text{where }
J=\begin{pmatrix}
 0 & i\\
 i & 0
  \end{pmatrix}.
\]
This enables us to re-find the hyperbolic analogue of the Schwarz reflection principle. As we have just seen in the proof of Lemma~\ref{lemma-p-q-real}, the system $\g\cdot S_i$ is real on $\left(t_i,t_{i+1} \right)$, and thus extends onto the lower half-plane $\C_-$ through $\left(t_i,t_{i+1} \right)$ by setting for every $x\in\C_-$
\[
 \g(x)\cdot S_i := \tau\big(\g\cdot S_i\big)(x).
\]
The function $\g\cdot S_i$ is then holomorphic on the simply connected open set
\[
 U_i = \C_+\cup\C_-\cup\left(t_i,t_{i+1} \right).
\]
We obtain by this way $n+3$ different analytic continuations $Y_1,\ldots, Y_{n+3}$ onto $\C_-$ of the Weierstrass data $Y_0=\g$. They satisfy
\[
 Y_i:U_i\to\C^2, \qquad Y_i\big|_{\C_+} = Y_0, \qquad Y_i\big|_{\C_-}=\tau\big(Y_0\cdot S_i\cdot \overline S_i^{-1}\big),
\]
and they involve $n+3$ different continuations $X_i:\C_-\to\R^3$ of the maxface $X$. The maxface $X_i$ of Weierstrass data $Y_i$ represents on $\C_-$ the same surface than the maxface of Weierstrass data $Y_0\cdot S_i\cdot \overline S_i^{-1} \cdot J$ defined on $\C_+$. From the identity~\eqref{SU}, we have
\[
 S_i\cdot \overline S_i^{-1} \cdot J = S_i \cdot J \cdot S_i^{-1}.
\]
Since the matrix $S_i\in\SU(1,1)$ is a pre-image by the universal cover~\eqref{spin-covering} of a direct isometry of $\L^3$ mapping the direction vector $u_i$ on $\pm e_1$, multiply the Weierstrass data $Y_0$ by $S_i \cdot J \cdot S_i^{-1}$ amounts to applying to the maxface $X$ the ``hyperbolic half-turn'' $H_i$ of direction $u_i$ defined by
\[
 H_i = \left(R_{S_i}\right)^{-1}
\begin{pmatrix}
 1 & 0 & 0\\
 0 &-1 & 0\\
 0 & 0 &-1\\
\end{pmatrix}
R_{S_i} \quad \in \SO^-(2,1).
\]
Let us define the set
\[
 \SU^-(1,1) :=  \SU(1,1)\cdot J
 = \left\{\left.
\begin{pmatrix}
 a & b\\
 -\overline b & -\overline a
\end{pmatrix}
\ \right\vert \ a,b\in \C,\ \  b\overline b - a\overline a = 1
\right\}
\]
(which is not a group). The covering $R :\SU(1,1) \to \SO^+(2,1)$ extends onto $\SU^-(1,1)$ such that  $R\left(\SU^-(1,1) \right) = \SO^-(2,1)$, and Proposition~\ref{prop-spin} is still true for matrices $A\in\SU^-(1,1)$. %(and still a double covering)
The identity~\eqref{SU} is transformed on $\SU^-(1,1)$ as follow:
\begin{equation}
 \forall A\in \SU^-(1,1) \qquad J A= -\overline A J.
\label{SU-}
\end{equation}
The matrices $A\in\SU^-(1,1)$ which are mapped by $R$ onto hyperbolic half-turns are characterized by the equation $A^2=-\I_2$, and the two pre-images of the same hyperbolic half-turn are the opposite one of the other.

\bigskip

The Schwarz reflection principle, which can thus be seen as a consequence of reality properties of the maxface $X$, enables us to determine how the fundamental solution $Y_0$ of Equation~\eqref{E} is transformed around the singular points $x=t_i$. We fix a base point $x_0\in\C_+$. The fundamental group $\pi_1\left(\P\ssm S(t), x_0\right)$ is generated by the equivalent classes of the loops $\ga_1,\ldots,\ga_{n+3}$ drawn in Figure~\ref{fig-ga-i}. We denote by $\ga_i\ast Y_0$ the analytic continuation of the fundamental solution $Y_0$ along the loop $\ga_i$. It is still a fundamental solution of Equation~\eqref{E} holomorphic on the upper half-plane $\C_+$, since the coefficients $p(x)$ and $q(x)$ are uniform around $x=t_i$. The monodromy matrix $M_i$ of $Y_0$ along the loops $\ga_i$ is then the unique invertible matrix satisfying for every $x\in\C_+$
\[
 \ga_i\ast Y_0(x) = Y_0(x)M_i.
\]
The matrices $M_1,\ldots,M_{n+3}$ satisfy
\[
M_{n+3}\cdots M_1 = \I_2,
\]
and they form a system of generators of the monodromy of Equation~\eqref{E}.

\begin{figure}
\centering
\begin{pspicture}(0,0)(14,8)
%\psgrid
\rput(7,4){\includegraphics[height=7cm]{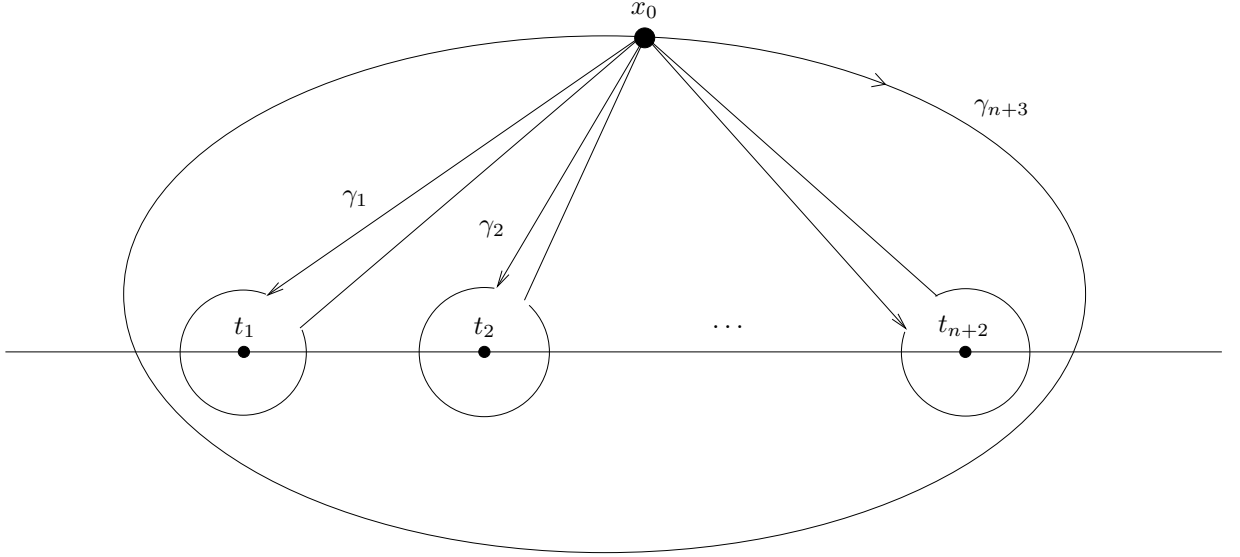}}
\rput(7.4,7.7){$x_0$}
\rput(12.1,6.4){$\ga_{n+3}$}
\rput(2.15,3.5){$t_1$}
\rput(5.3,3.5){$t_2$}
\rput(11.6,3.5){$t_{n+2}$}
\rput(8.5,3.5){$\ldots$}
\rput(3.6,5.2){$\ga_1$}
\rput(5.4,4.8){$\ga_2$}
\end{pspicture}
\caption{The loops $\ga_i$}
\label{fig-ga-i}
\end{figure}

\begin{prop}
For every $i=1,\ldots,n+3$, we can choose a pre-image in $\SU^-(1,1)$ of the hyperbolic half-turn of direction $D_i$, which we will also denote by $D_i\in\SU^-(1,1)$, such that the monodromy matrices $M_i$ of the fundamental solution $Y_0$ around the singularities $x=t_i$ write
\begin{equation}
    M_i = D_iD_{i-1}^{-1}.
\label{def1-Mi}
\end{equation}
The matrices $M_i$ thus belong to $\SU(1,1)$, and they respectively correspond to the direct isometries of timelike axis $v_i$ and of angle $2\pi\t_i$.
\label{prop-mono}
\end{prop}

\begin{proof}
The monodromy matrix $M_i$ is the unique matrix satisfying for all $x\in\C_-$: $ Y_{i-1}(x) = Y_i(x) M_i$. When we extend $Y_0$ along the loop $\ga_i$, we apply to the image of $X$ two successive hyperbolic half-turns, and the matrix $M_i$ thus writes
\[
 M_i = \w D_i \w D_{i-1},
\]
where  the $\w D_j$ are pre-images in $\SU^-(1,1)$ of hyperbolic half-turns of direction $u_j$. We want to compare the pre-images arising in two successive monodromy matrices $M_{i-1}$ and $M_i$. The two pre-images of the hyperbolic half-turn of direction $u_i$ are $\pm S_iJS_i^{-1}$. For every $i=1,\ldots,n+3$, we fix $D_i$ to be
\[
D_i = S_iJS_i^{-1}.
\]
The matrix $D_i$ depends on $S_i$, which is any pre-image of any direct isometry mapping $u_i$ into $\e e_1$, only through the sign $\e=\pm1$. This sign is fixed once we ask for $Y_0S_i$ to be real on $(t_i,t_{i+1})$. The matrix $D_i$ satisfies on $\C_-$
\[
 Y_i=\tau\big(Y_0 D_i J^{-1}\big).
\]
From the identity~\eqref{SU-}, we have $\left(\overline D_i\overline J^{-1}\right)^{-1} = D_i J^{-1}$, and we obtain on $\C_+$:
\[
 Y_0 = \tau\big(Y_i D_i J^{-1}\big).
\]
Thus for every $x\in\C_-$
\[
 Y_{i-1}(x)D_{i-1} = Y_i(x)D_i
\]
\ie $ Y_{i-1}(x) = Y_i(x) D_iD_{i-1}^{-1}$, which ends the proof.
\end{proof}

In the previous proof, we could also have chosen for every $i$, $D_i=-S_iJS_i^{-1}$. But if we consider at a vertex $a_i$ the two matrices $D'_{i-1}=S_{i-1}JS_{i-1}^{-1}$ and $D'_i=-S_iJS_i^{-1}$, then the product $D'_i(D'_{i-1})^{-1}=-M_i$ is still a pre-image of the same direct isometry than $M_i$, but their eigenvalues are different. From the study of the local behavior of the maxface $X$ at $x=t_i$ (see the next subsection), we know
that it amounts to changing the orientations of the edges number $i-1$ and $i$ one with respect to the other (it changes the normal vector $v_i$ into $-v_i$ and the exterior angle $\t_i$ into $1-\t_i$). We thus can say that the choice of the pre-image of a hyperbolic half-turn around an edge of the polygon $P$ is induced by the orientations of the directions of $P$, up to a global change of $D$ into $-D$.

From the explicit expression of the universal cover~\eqref{spin-covering}, we can infer the expression of the matrices $M_j$
\[
M_j = \cos\left( \pi\t_j \right) \I_2 + \sin \left( \pi\t_j \right)
\begin{pmatrix}
iv_3^j & v_1^j+iv_2^j\\
v_1^j-iv_2^j & -iv_3^j
\end{pmatrix},
\]
where $v_j=(v_1^j,v_2^j,v_3^j)$, and where the choice of the pre-image is induced by the previous considerations on the orientation. But their expression as product of hyperbolic half-turns will be more useful.

\subsection{Local behavior around the singularities}

The local behavior of the Weierstrass data $\g$ around the vertices of the polygonal boundary curve and around the umbilics is exactly the same as in the Euclidean case. We briefly remind here how we obtain it from the local expression of the maxface $X$ at these points. We only consider what we call \emph{the generic situation}, we explain after how it involves any general ones (see Propositions 2.9 and 2.12 in~\cite{Desideri} for more details).

\bigskip

The Weierstrass data $G$ and $H$ have no essential singularity at
the points $x=t_i$ (even at $x=\infty$ when the maxface $X$ has a
helicoidal end), which are thus Fuchsian singularities of
Equation~\eqref{E}. Since the umbilics are obviously Fuchsian
singularities too, we obtain the following proposition.

\begin{prop}
Equation~\eqref{E} is Fuchsian on the Riemann sphere $\P$.
\end{prop}

The eigenvalues of the monodromy matrix $M_i$ determine the
exponents at the singular point  $x=t_i$ up to integers. Their
exact values are given by the local expansion of the maxface $X$
at $x=t_i$. As we have seen in Section~\ref{section-max}, the
angle of $X$ at the vertex $a_i$ is $(1-\t_i)\pi\in(0,\pi)$ or
$(1+\t_i)\pi\in(\pi,2\pi)$, where $\t_i\pi$ is the exterior angle
of the polygonal boundary curve $P$ at $a_i$. Up to a direct
isometry of $\L^3$, we can always suppose that the Gauss map
$N(t_i)$ is equal to the third-coordinate vector $e_3=(0,0,1)$ or
to its opposite, and that the direction vector $u_i$ is equal to
the first-coordinate vector $e_1$. In this position we have the
following expansion at $x=t_i$
\[
     X(x)-X(t_i) \sim \Re
        \begin{pmatrix}
            a(x-t_i)^{1-\e_i\t_i}\\
            ib(x-t_i)^{1-\e_i\t_i}\\
            ic(x-t_i)^{1+r_i}\\
        \end{pmatrix},
\]
where $r_i$ is a non-negative integer, $\e_i=\pm1$ such that
$\e_i=+1$ when $r_i=0$, and $a$, $b$ and $c$ are non-zero real
constants. We see that $\mid \dd X/ \dd x \mid$ tends to $0$ at
$x=t_i$ if and only if $\e_i=-1$. This means that the point
$x=t_i$ is a boundary branch point of $X$ if and only if the angle
of $X$ at the vertex $a_i$ is $(1+\t_i)\pi$. If so, the branch
point order is $1$.

The generic situation corresponds to $r_i=0$. In this case, the
Weierstrass data have the following local behavior
\[
 G(x) \sim \a(x-t_i)^{-\frac{\t_i}{2}}, \qquad H(x) \sim \b(x-t_i)^\frac{\t_i}{2},
\]
(where $\a$ and $\b$ are non-zero constants), or the converse
whether $N(t_i)=e_3$ or $N(t_i)=-e_3$. At $x=\infty$, we obtain in
the generic case
\[
 G(x) \sim \a\left(\frac1x\right)^{1-\frac{\t_i}{2}}, \qquad H(x) \sim
 \b\left(\frac1x\right)^\frac{\t_i}{2},
\]
(or the converse). Since we then have $G(x)H(x)\sim \a\b/x$, we
see that the generic situation leads to a helicoidal end at
$a_{n+3}$.

\bigskip

Let us denote by $\l_1,\ldots,\l_N$ ($N\in \N$) the umbilics of
$X$, and their conjugates in $\C_-$. By the same way, one can prove
that the exponents at an umbilic $x=\l_k$ are $0$ and an integer
$m_k\geq2$ such that $m_k-1$ is the order of the zero of the Hopf
differential at $x=\l_k$ (again by assuming $N(\l_k)=\pm e_3$).
The generic situation corresponds to $m_k=2$. The general
situation is obtained by merging such generic umbilics together.
Since the Weierstrass data $\g$ are holomorphic at the umbilics,
they are apparent singularities, and there are no other apparent
singularities than umbilics. From Lemma~\ref{lemma-p-q-real}, we
know that the apparent singularities of Equation~\eqref{E} are
real or conjugate by pairs.

\bigskip

We can consider that a non-generic singularity $t_i$ or $\l_k$ is obtain from a generic one by merging it with apparent singularities. This process increases by integers one of the exponents at the generic singularity. For example, if the polygonal boundary curve is closed, it means that an apparent singularity coincides with the singular point $x=\infty$: it changes the exponents at infinity from $(1-\frac{\t_i}{2}, \frac{\t_i}{2})$ to $(1-\frac{\t_i}{2},1+\frac{\t_i}{2})$. In the point of view of Equation~\eqref{E}, this merging process has to be done carefully, since Equation~\eqref{E} is not canonical (for example it is not $\SL(2,\C)$-type). The use of Fuchsian systems will make this process clearer.

In general situations, the integers $r_i$, $m_k$ and the number
$N$ of apparent singularities are not entirely free: they are
related together by the Fuchs relation~\eqref{Fuchs}. It implies
in particular that the number $N$ is smaller or equal to $n$. The
equality $N=n$ holds if and only of all the singularities are
generic.

\begin{rem}
A maxface with a \emph{closed} polygonal boundary curve has at most $n-1$ apparent singularities. This is the reason why we consider possibly unclosed polygonal curves. Indeed, to get a positive answer to the Riemann--Hilbert problem and to construct isomonodromic deformations with $n+3$ non-apparent singularities, we have to authorize $n$ apparent singularities (see Ohtsuki~\cite{Oht}).
\end{rem}

\subsection{The space of Fuchsian equations}

Let us fix an oriented direction $D\in\D^n$, and denote by
$\t_i\pi$ the exterior angles of $D$, such that $\t_i\in(0,1)$. We
identify every oriented direction $D_i$ with the appropriate
pre-image of the hyperbolic half-turn of direction $D_i$ given by
Proposition~\ref{prop-mono}. We have seen that any
equation~\eqref{E} associated by the Weierstrass representation
with a maxface $X\in\X^n_D$ satisfies the following properties.

\begin{enumerate}
 \item \label{cond-SdR} Equation~\eqref{E} is Fuchsian on the Riemann sphere $\P$. It has $n+3$ non-apparent distinct singularities $t_1,\ldots,t_n,\ t_{n+1}=0$, $t_{n+2}=1$, $t_{n+3}=\infty$, and $n$ apparent ones $\l_1,\dotsc,\l_n$. Its exponents are given by:
\begin{equation}
\begin{split}
            &\begin{pmatrix}
                    x=t_i            & x=\infty               & x=\l_k\\
                    -\frac{\t_i}{2}  & \frac{\t_\infty}{2}    & 0\\
                    \frac{\t_i}{2}   & 1-\frac{\t_\infty}{2}  & 2
             \end{pmatrix}\\
            &\begin{array}{ccc}
                    i=1,\dotsc,n+2, &   & k=1,\dotsc,n
             \end{array}.
\end{split}
\label{SdR-E}
\end{equation}
 \item \label{cond-mono} A system of generators $M_i$ ($i=1,\dotsc,n+3$) of the monodromy of Equation~\eqref{E} along the loops $\ga_i$ belongs to the group $\SU(1,1)$ and writes
\[
 M_i=D_iD_{i-1}^{-1}, \quad \text{ where } D_i\in\SU^-(1,1),\ D_i^2=-\I_2.
\]
 \item \label{cond-realite} Equation~\eqref{E} is real, and the $n$-tuple of singular points $t=(t_1,\ldots,t_n)$ belongs to the simplex
\[
 \pi^n=\left\{ (t_1,\ldots,t_n)\in\R^n\ \big|\ t_1<\cdots<t_n<0\right\}  .
\]
\end{enumerate}

\begin{nota}
For any oriented direction $D\in\D^n$, we denote by $\E^n_D$ the set of equations satisfying conditions~\ref{cond-SdR}, \ref{cond-mono} and~\ref{cond-realite} above, where the direction $D$ (and thus also the $\t_i$'s) are fixed.
\end{nota}

Notice that the only difference between this space and the analogous space of equations associated with minimal disks with a polygonal boundary curve is that the matrices $D_i$ belong to $\SU^-(1,1)$ instead of $\SU(2)$.

The coefficients $p(x)$ and $q(x)$ of an equation~\eqref{E} satisfying condition~\ref{cond-SdR} are of the form
\begin{align*}
            p(x) &= \sum_{i=1}^{n+2} \dfrac{1}{x-t_i} - \sum_{k=1}^n \dfrac{1}{x-\l_k},\\
            q(x) &= -\frac14 \sum_{i=1}^{n+2} \dfrac{\t_i^2}{(x-t_i)^2} + \dfrac{\kappa}{x(x-1)} - \sum_{i=1}^n\dfrac{t_i(t_i-1)K_i}{x(x-1)(x-t_i)}\\
               & \qquad   + \sum_{k=1}^n\dfrac{\l_k(\l_k-1)\mu_k}{x(x-1)(x-\l_k)},
\end{align*}
where $\kappa=\frac{\t_\infty}{2}\left(1- \frac{\t_\infty}{2}\right) +\frac14 \sum_{i=1}^{n+2} \t_i^2 $. But all the possible choices of the parameters $t$, $K$, $\l$ and $\mu$ do not necessarily define an equation satisfying condition~\ref{cond-SdR}, since it could have logarithmic singularities at the $\l_k$. The Hopf differential of a maxface $X\in\X^n_D$ thus writes
 \[
 Q = i\frac{\Lambda(x)}{T(x)} \, \dd x^2,
\]
with
\begin{equation}
 \Lambda(x)=\prod_{k=1}^n(x-\l_k), \quad T(x)=\prod_{i=1}^{n+2}(x-t_i)
\label{def-T-L}
\end{equation}
where, as before, some apparent singularities $\l_k$ can coincide ones with others, or with a point $t_i$.

The space of equations $\E^n_D$ provides an appropriate description of the space of maxfaces $\X^n_D$, as stated in the following result.

\begin{prop}
The correspondence given by the spinor Weierstrass representation between the spaces $\X^n_D$ and $\E^n_D$ is one-to-one.
\end{prop}

This means that any equation in the space $\E^n_D$ admits a fundamental solution $\g$ which is the Weierstrass data of a maxface $X\in\X^n_D$. We do not reproduce the proof of this result, since it is the same as the one of Proposition 2.15 in~\cite{Desideri} for minimal disks. The idea is to consider all fundamental solutions having the matrices $M_i$ of condition~\ref{cond-mono} for monodromy matrices: they form a one-parameter family $\l Y(x)$, $\l\in\C^*$. We mainly have to prove that we can choose the scalar $\l$ in such a way that for every $i=1,\ldots,n+3$ there is a matrix $S_i\in\SU(1,1)$ such that the solution $\l Y(x)\cdot S_i$ is real on $\left(t_i,t_{i+1}\right)$. We need for this purpose the reality condition~\ref{cond-realite}, but also the particular expression of the matrices $M_i$. The injective nature of the correspondence comes from the fact that in a given associated family of maxfaces, at most one can be bounded by a polygon.

%%%%%%%%%%%%%%%%%%%%%%%%%%%%%%%%%%%%%%%%%%%%%%%%

\section{Isomonodromic deformations}
\label{section-isomono}

We intend in this section to describe the space $\E^n_D$, in order to provide an explicit description of the space $\X^n_D$ and thus also of the length ratios of their polygonal boundary curves. The Garnier system involves isomonodromic deformations of Fuchsian equations satisfying condition~\ref{cond-SdR}, and can be used to express the equations in $\E^n_D$. Unfortunately, the Garnier system does not have the Painlev\'e property (see Definition~\ref{def-prop-P}). This is the main reason why we will use, unlike Garnier, Fuchsian \emph{systems} instead of equations. This choice also simplifies the resolution in many other ways, since systems are in a sense more canonical than equations.

\subsection{The corresponding space of Fuchsian systems}

We first remind generalities about Fuchsian systems and their relations with Fuchsian equations.

Let us consider a first-order $2\times2$ linear differential system
\begin{equation}
 Y'=A(x)Y
\label{A0}\tag{$A_0$}
\end{equation}
where the function $A(x)$ is meromorphic on the Riemann sphere $\P$, with values in $\MM(2,\C)$. System~\eqref{A0} is said \emph{Fuchsian} if the poles of $A(x)$ are simple. We suppose that its singularities are $t_1,\ldots,t_n,\ t_{n+1}=0,\ t_{n+2}=1,\ t_{n+3}=\infty$, and it thus writes
\[
 A(x)=\sum_{i=1}^{n+2}\frac{A_i}{x-t_i}.
\]
Since we assume that the point $x=\infty$ is singular, the residue matrix
\[
 A_\infty:=-\sum_{i=1}^{n+2}A_i
\]
is non-zero (for notational simplicity, we sometimes write $A_{n+3}$ in place of $A_\infty$). We call a \emph{fundamental solution} of System~\eqref{A0} a matrix $\Y(x)$ whose colons form a basis of the $2$-dimensional vector space of all solutions of System~\eqref{A0}. A fundamental solution is invertible and holomorphic on the universal covering space of $\P\ssm S(t)$. We only consider Fuchsian systems satisfying the two following assumptions:
\begin{itemize}
 \item System~\eqref{A0} is \emph{non-resonant}: for every $i=1,\ldots,n+3$, the eigenvalues $\t_i^+$ and $\t_i^-$ of the residue matrix $A_i$ satisfy $\t_i^+-\t_i^- \notin \Z$,
 \item System~\eqref{A0} is \emph{normalized at infinity}:
\begin{equation}
 A_\infty=
    \begin{pmatrix}
     \t_\infty^+ & 0\\
     0 & \t_\infty^-\\
    \end{pmatrix}.
\label{A_infty}
\end{equation}
\end{itemize}
Since System~\eqref{A0} is non-resonant, its singularities are non-logarithmic, and there is at each point $x=t_i$ a fundamental solution of the form
\[
 R_i(x) (x-t_i)^{L_i}, \qquad \text{with }
 L_i=
\begin{pmatrix}
 \t_i^+&0\\
 0&\t_i^-
\end{pmatrix},
\]
where the matrix $R_i(x)$ is holomorphic and invertible at $x=t_i$, and $R_i(t_i) \in\GL(2,\C)$ satisfies
\[
 A_i = R_i(t_i)L_iR_i(t_i)^{-1}.
\]
Such a fundamental solution is said \emph{canonical} at $x=t_i$, since its monodromy matrix at this point is diagonal:
$
 \begin{pmatrix}
  e^{2i\pi\t_i^+} & 0\\
  0 & e^{2i\pi\t_i^-}
 \end{pmatrix}.
$

\noindent
Since System~\eqref{A0} is normalized at $x=\infty$, there is a unique canonical solution at $x=\infty$ of the form
\begin{equation}
 \Y_\infty(x)= R_\infty\left(\frac1x\right) x^{-L_\infty}, \qquad \text{and } L_\infty = A_\infty,
\label{Y_infty}
\end{equation}
where the matrix $R_\infty(w)$ is holomorphic at $w=0$ and $R_\infty(0)=\I_2$.

If we consider a first-order $2\times2$ linear differential system
\begin{equation}
 Y'=A(x)Y, \qquad A(x)=
    \begin{pmatrix}
     A_{11}(x) & A_{12}(x)\\
     A_{21}(x) & A_{22}(x)
    \end{pmatrix},
\label{sysqcq}
\end{equation}
where the functions $A_{ij}(x)$ are meromorphic on the Riemann sphere and $A_{12}(x)$ does not vanish identically, we obtain by a direct computation that the first component $y_1$ of any solution $Y=(y_1,y_2)^t$ of the system~\eqref{sysqcq} satisfies the second-order differential equation
\begin{equation}
  y''+p(x)y'+q(x)y=0,
\label{eqcq}
\end{equation}
where
\begin{align*}
 p(x)&=-\frac{A'_{12}(x)}{A_{12}(x)}-\tr A(x)\\
 q(x)&=-A'_{11}(x)+A_{11}(x)\frac{A'_{12}(x)}{A_{12}(x)}+\det A(x).
\end{align*}
It is thus obvious that if the system~\eqref{sysqcq} is Fuchsian, then its associated equation~\eqref{eqcq} is also Fuchsian. Moreover, if $x=\l$ is a zero of $A_{12}(x)$ of order $m$, but is not a singularity of the system~\eqref{sysqcq}, then $x=\l$ is an apparent singularity of the equation~\eqref{eqcq} with exponents $0$ and $m+1$. Let us consider now the Fuchsian equation associated with the Fuchsian system~\eqref{A0}. From the normalization at infinity of System~\eqref{A0}, we know that its coefficient $A_{12}(x)$ has $n$ zeros counted with multiplicity, that we denote by $\l_1,\ldots,\l_n$. Then we have
\[
 A_{12}(x) = \xi\frac{\Lambda(x)}{T(x)},
\]
where $\xi=\sum_{i=1}^{n+2}t_iA^i_{12}$, and the polynomials $T(x)$ and $\Lambda(x)$ are defined by~\eqref{def-T-L}. The exponents of the Fuchsian equation associated with~\eqref{A0} are thus given by:
\begin{equation}
\begin{split}
            &\left(
         \begin{array}{ccc}
                    x=t_i   & x=\infty       & x=\l_k\\
                    \t_i^+  & \t_\infty^+    & 0\\
                    \t_i^-  & \t_\infty^-+1  & 2\\
             \end{array}
         \right)\\
            &\begin{array}{ccc}
                    i=1,\dotsc,n+2, &   & k=1,\dotsc,n,
             \end{array}
\end{split}
\label{(SdR)}
\end{equation}
and the equation has no logarithmic singularity.

\bigskip

This correspondence between Fuchsian systems and Fuchsian equations enables us to define the space of Fuchsian systems associated with a maxface $X$ in $\X^n_D$. We can see that we have two possible choices for the normalization at infinity. The appropriate one is the following.

\begin{nota}
For any oriented direction $D\in\D^n$, we denote by $\A^n_D$ the space of first-order $2\times2$ Fuchsian systems whose associated equation belongs to the space $\E^n_D$, and which are normalized at infinity by
\[
 A_\infty= \left( 1-\tfrac{\t_\infty}{2} \right)
    \begin{pmatrix}
     1& 0\\
     0 & -1\\
    \end{pmatrix}.
\]
\end{nota}

Of course the correspondence between $\A^n_D$ and $\X^n_D$ is no longer one-to-one, since different systems may define the same equation.  The correspondence between Fuchsian systems and maxfaces is not so natural than the one between scalar equations and maxfaces, since there are also non-Fuchsian systems defining Fuchsian equations in $\E^n_D$, and thus maxfaces in $\X^n_D$ as well. Consider for example the non-Fuchsian system
\[
 Y' =
 \begin{pmatrix}
 0 & 1\\
 -q(x) & -p(x)\\
 \end{pmatrix}
 Y.
\]
Whereas the space $\E^n_D$ is entirely given by the Weierstrass representation, we choose the space $\A^n_D$ because it is convenient to describe the space $\X^n_D$.

To describe the space $\A^n_D$, we need to make the converse operation more explicit: in the non-resonant case, it is known that we can explicitly describe the set of Fuchsian systems defining a given Fuchsian equation. We have seen that the coefficient $A_{12}(x)$ of System~\eqref{A0} is entirely determined by the parameters of its associated equation, and by an additional parameter $\xi\in\C^*$. Actually, this is also true for the other coefficients. In~\cite{Desideri}, we established the following proposition (see also~\cite{IKSY}).

\begin{prop}
Let $(E)$ be a second-order Fuchsian equation with exponents~\eqref{(SdR)} and without logarithmic singularity. The set of non-resonant Fuchsian systems normalized at infinity by~\eqref{A_infty} and defining Equation $(E)$ is a one-parameter family
\[
Y' = A_\xi(x) Y,
\]
where $\xi\in\C^*$. Moreover,
\[
 A_\xi(x) =
    \begin{pmatrix}
    A_{11}^0(x)         & \xi A_{12}^0(x)\\
    \dfrac{1}{\xi}A_{21}^0(x) & A_{22}^0(x)
    \end{pmatrix}
\]
where the matrix $\left(A_{ij}^0(x)\right)_{i,j}$ is explicitly determined by Equation $(E)$.
\label{prop-eq->sys}
\end{prop}

Thanks to Proposition~\ref{prop-eq->sys}, we can characterize the elements of $\A^n_D$ by translating separately each of the conditions~\ref{cond-SdR}, \ref{cond-mono} and \ref{cond-realite} in terms of Fuchsian systems. The only point that requires additional work is the reality condition~\ref{cond-realite}. System~\eqref{A0} defines an equation satisfying \ref{cond-realite} if and only if its singularities are real: $t\in\pi^n$, and if it defines the same equation than its conjugate system $\left(\tau(A_0)\right)$, which is given by
\begin{equation}
  Y'=\tau(A)(x)Y, \qquad
  \tau(A)(x)=\sum_{i=1}^{n+2}\frac{\overline{A}_i}{x-t_i}.
\tag{$\tau (A_0)$}\label{sys-conj}
\end{equation}
By Proposition~\ref{prop-eq->sys}, this means that both systems belong to the same family, \ie there exists $\xi\in\C^*$ such that for every $i=1,\ldots,n+2$, we have
\[
 \overline{A}_i=
    \begin{pmatrix}
     A_{11}^i               & \xi A_{12}^i\\
     \frac1\xi A_{21}^i & A_{22}^i
    \end{pmatrix}.
\]
This provides the desired characterization. For every direction $D\in\D^n$, a first-order $2\times2$ differential system~\eqref{A} belongs to the space $\A^n_D$ if and only if it satisfies the three following conditions.

\begin{enumerate}[label=(\alph{*}),ref=(\alph{*})]
 \item \label{cond-sys-SdR} System~\eqref{A} is Fuchsian, with $n+3$ singularities $t_1,\ldots,t_n$, $t_{n+1}=0$, $t_{n+2}=1$, $t_{n+3}=\infty$. It thus writes:
\begin{equation}
Y'=A(x)Y, \qquad A(x)=\sum_{i=1}^{n+2}\frac{A_i}{x-t_i}.
\tag{$A$}\label{A}
\end{equation}
The eigenvalues of the residue matrix $A_i$ are $\frac{\t_i}{2}$ and $-\frac{\t_i}{2}$ ($i=1,\ldots,n+2$), and it is normalized at infinity by:
$
A_\infty = \left( 1-\frac{\t_\infty}{2} \right)
\begin{pmatrix}
1&0\\
0&-1
\end{pmatrix}.
$
 \item \label{cond-sys-mono} A system of generators $M_i$ ($i=1,\dotsc,n+3$) of the monodromy of System~\eqref{A} along the loops $\ga_i$ belongs to the group $\SU(1,1)$ and writes
\[
 M_i=D_iD_{i-1}^{-1}, \quad \text{ where } D_i\in\SU^-(1,1),\ D_i^2=-\I_2.
\]
 \item \label{cond-sys-realite} The $n$-tuple of singular points $t=(t_1,\ldots,t_n)$ belongs to the simplex $\pi^n$, and there is a real number $\eta$ such that for every $i=1,\ldots,n+2$ the residue matrix $A_i$ is given by
\[
  A_i=
\begin{pmatrix}
 a_i&b_ie^{i\eta}\\
 c_ie^{-i\eta}&-a_i
\end{pmatrix}
\qquad \text{where } a_i \in\R \text{ and } b_i, c_i\in [0,+\infty).
\]
\end{enumerate}

%The exponents $\t_i$ denote again the exterior angles associated with the oriented direction $D$.
We intend, by means of isomonodromic deformations, to obtain an explicit description of the space $\A^n_D$. But we have to deal first with the reality condition~\ref{cond-sys-realite}.

\subsection{The reality condition}

As we have already mentioned during the study of the equation associated with a maxface in $\X^n_D$, reality properties and monodromy are closely related. In~\cite{Desideri} (Proposition 3.13), we obtained for non-resonant Fuchsian systems the following characterization of the reality condition~\ref{cond-sys-realite} by the monodromy.

\begin{prop}
Suppose that the singularities $t_i$ of the non-resonant Fuchsian System~\eqref{A0} are real, that the eigenvalues $\t_i^+$ and $\t_i^-$ are real or conjugate ($i=1,\ldots,n+2$) and the eigenvalues $\t_\infty^+$ and $\t_\infty^-$ are real.

Then System~\eqref{A0} satisfies the reality condition~\ref{cond-sys-realite} if and only if for every system of generators $\left(M_1,\ldots,M_{n+3}\right)$ of the monodromy along the loops $\ga_i$, there exists a matrix $C\in \GL_2(\C)$ such that
\begin{equation}
 C^{-1}\overline{M_i}C = (M_i\ldots M_1)^{-1} M_i^{-1} (M_j\ldots M_1)
\label{C1}
\end{equation}
($i=1,\ldots,n+3$). We call this condition \emph{Condition \textbf{C1}}.
\label{prop-sys-reel-1}
\end{prop}

We then proved the following result in the case of a unitarizable monodromy (see Proposition 3.14 in~\cite{Desideri}).

\begin{prop}
Under the same assumptions as in Proposition~\ref{prop-sys-reel-1}, if a system of generators $\left(M_1,\ldots,M_{n+3}\right)$ of the monodromy of System~\eqref{A0} is contained in $\SU(2)$ or in $\SU(1,1)$, then System~\eqref{A0} satisfies the reality condition~\ref{cond-sys-realite} if and only if there exist $n+3$ invertible matrices $D_1,\ldots,D_{n+3}$ such that
\[
\begin{cases}
     M_i=D_iD_{i-1}^{-1} \qquad (i=1,\ldots,n+3)\\
     {D_1}^2 =\cdots= {D_{n+3}}^2.
\end{cases}
\]
We call this condition \emph{Condition \textbf{C2}}.
\label{prop-sys-reel-2}
\end{prop}

Condition~\ref{cond-sys-realite} is thus a consequence of conditions~\ref{cond-sys-SdR} and \ref{cond-sys-mono}, and these two conditions then entirely characterize the space $\A^n_D$.

\subsection{The Schlesinger system}

We now briefly recall how the Schlesinger system provides isomonodromic deformations of non-resonant Fuchsian systems.

Let $\B^n$ be the open subset of $\C^n$ defined by
\begin{equation}
 \B^n= \left\{ (t_1,\ldots,t_n) \in (\C \ssm \{0,1\})^n \quad |\quad \forall i\neq j \quad t_i \neq t_j \right\},
 \label{def-B}
\end{equation}
and let $U$ be a simply connected open subset of $\B^n$. Let us consider a Fuchsian system which analytically depends on a parameter $t\in U$
\begin{equation}
Y'=A(x,t)Y, \qquad A(x,t)=\sum_{i=1}^{n+2}\frac{A_i(t)}{x-t_i},\qquad t\in U.
\label{At}\tag{$A_t$}
\end{equation}
We assume System~\eqref{At} to be non-resonant and normalized at infinity, and the eigenvalues $\t_i^+$ and $\t_i^-$ of the matrices $A_i(t)$ to be independent of $t$. We denote by $\Y_\infty(x,t)$ the unique fundamental solution~\eqref{Y_infty} of~\eqref{At} which is canonical at infinity. For sufficiently small variations of $t$, we can choose a base point $x_0\in\P\ssm S(t)$ independent of $t$, and we can consider that the fundamental group $\pi_1\left( \P\ssm S(t), x_0 \right)$ does not depend on $t$. The monodromy of System~\eqref{At} is then well-defined.

\begin{defn}
 The \emph{Schlesinger system} is the following system of nonlinear differential equations
\begin{equation}
 \dd A_i = \sum_{\substack{j=1\\ j\neq i}}^{n+2} [A_j,A_i] \dd\log(t_i-t_j),  \qquad i=1,\ldots,n+2,
\label{schlesinger}
\end{equation}
where $\dd$ denotes the exterior differentiation with respect to $t$.
\end{defn}

We then have the following well-known results.

\begin{thm}
\begin{enumerate}
\item The Schlesinger system~\eqref{schlesinger} is completely integrable.
\item The monodromy group of the fundamental solution $\Y_\infty(x,t)$ of System~\eqref{At} is independent of $t$ if and only if the matrices $A_i(t)$, $i=1,\ldots,n+2$, satisfy the Schlesinger system~\eqref{schlesinger}.
\item (\cite{Malgrange}, \cite{Miwa}) The Schlesinger system~\eqref{schlesinger} has the Painlev\'e property (see Definition~\ref{def-prop-P}). Moreover, any solution of the system is meromorphic on the universal covering space of $\B^n$.
\end{enumerate}
\label{thm-schlesinger}
\end{thm}

As already mentioned, we do not need an explicit description of the entire space $\A^n_D$, since it is ``larger'' than the space $\X^n_D$: we want to describe a part of it, which should be in bijection with $\E^n_D$, and thus also with $\X^n_D$. We construct it as follow. We fix an arbitrary point $t^0\in\pi^n$, and we consider a Fuchsian system ($A_0$) whose monodromy is given by the oriented direction $D\in\D^n$ by condition~\ref{cond-sys-mono}, and whose position of singularities is given by $t^0$. Such a system always exists, since for $2\times2$ systems, the Riemann--Hilbert problem always gets a positive answer (see Anosov and Bolibruch~\cite{AB}, or Beauville~\cite{Beauville} for a shorter exposition of the known results on the Riemann--Hilbert problem). We can always choose the system ($A_0$) to be normalized at infinity.  From the integrability of the Schlesinger system~\eqref{schlesinger}, we obtain an isomonodromic family of Fuchsian systems $\left( A_D(t) , t\in U \right)$ described by the Schlesinger system, such that $(A_D(t^0))=(A_0)$, and where the open set $U\subset\B^n$ is a simply connected neighborhood of the simplex $\pi^n$. From Proposition~\ref{prop-sys-reel-2}, we can then deduce that
\[
\left( A_D(t) , t\in\pi^n \right)  \subset  \A^n_D.
\]
Any possible choice for the solution ($A_0$) of the Riemann--Hilbert problem leads by this way to an isomonodromic family of Fuchsian systems included in the space $\A^n_D$ (when $t\in\pi^n$), and obviously every element  of  $\A^n_D$ belongs to such a family.

If we consider two of these families $\left( A^1_D(t) , t\in U \right)$ and $\left( A^2_D(t) , t\in U \right)$, then we can easily see that, for all $t\in U$, the Fuchsian systems $\left( A^1_D(t)\right)$ and $\left( A^2_D(t)\right)$ define the same Fuchsian equation, denoted by $\left( E_D(t)\right)$: they correspond to different values of the parameter $\xi\in\C^*$ introduced in Proposition~\ref{prop-eq->sys} (see Lemma 3.12 in~\cite{Desideri}). The isomonodromic family of Fuchsian equations $\left( E_D(t) , t\in\pi^n \right)$ thus entirely describes the space $\E^n_D$, which is then parametrized by $t$. Actually, the family $\left( E_D(t) , t\in\pi^n \right)$ can also be defined via the Garnier system $\left(\mathcal G_n\right)$, but we do not use this point of view.

Finally, we arbitrarily fix an isomonodromic family $\left( A_D(t) , t\in\pi^n \right)$, given by a submanifold of an integral manifold of the Schlesinger system. It describes the space of maxfaces $\X^n_D$ as follow. By definition of $\A^n_D$, for all $t\in\pi^n$, there exists a fundamental solution $\Y_0(x,t)$ of $\left( A_D(t)\right)$ whose first line $\gxt$ is the Weierstrass data of a maxface in $\X^n_D$, denoted by $X_D(t) $, and we have
\[
\left( X_D(t) , t\in\pi^n \right)  = \X^n_D.
\]
We denote by $P_D(t)\in\p^n_D$ the polygonal boundary curve of the image of the maxface $X_D(t)$. The family $\left( P_D(t) , t\in\pi^n \right)$ is exactly the family of all polygonal curves of oriented direction $D$ that bound at least one maxface of disk-type.

Since the solution $\Y_0(x,t)$ is defined up to multiplication by real scalars, and since both $\Y_0(x,t)$ and the canonical solution $\Y_\infty(x,t)$ are $M$-invariant (\ie their monodromy group is independent of $t$), we know that there is a matrix $C_0\in\GL(2,\C)$ independent of $t$ such that
\[
\Y_0(x,t) = \Y_\infty(x,t) \cdot C_0.
\]
This enables us to study the behavior in $t$ of the maxfaces $X_D(t)$, and of the length ratios of their polygonal boundary curves $P_D(t)$.

\bigskip

From Theorem~\ref{thm-schlesinger}, (iii), we know that the residue matrices $A_1(t),\ldots,A_{n+2}(t)$ of the Fuchsian systems $\left( A_D(t) , t\in U \right)$ are meromorphic on $U$. Thanks to Proposition~\ref{prop-sys-reel-1}, we established in~\cite{Desideri} the holomorphicity at the real values of $t$ of the solutions of the Schlesinger system satisfying the reality condition~\ref{cond-sys-realite}, and thus in particular, of the family $\left( A_D(t) , t\in\pi^n \right)$.

\begin{prop}
Assume that the residue matrices $A_1(t),\ldots,A_{n+2}(t)$ of System~\eqref{At} satisfy the Schlesinger system~\eqref{schlesinger}, that the eigenvalues $\t_i^\pm$ are real or conjugate, and the eigenvalues $\t_\infty^\pm$ are real. If there exists a value $t^0\in\pi^n$ such that the monodromy of System~$(A_{t^0})$ satisfies Condition~\textbf{C1}, then the matrices $A_i(t)$ are holomorphic on a simply connected open neighborhood $U\subset\B^n$ of $\pi^n$.
\label{prop-Ai-holo}
\end{prop}

%%%%%%%%%%%%%%%%%%%%%%%%%%%%%%%%%%%%%%%%%%%%%%%%

\section{The length-ratio function}
\label{section-ratio}

The goal of this section, that ends the proof of Theorem~\ref{thm-plateau-max}, is to show that every polygonal curve in $\p^n_D$ belongs to the family $\left( P_D(t) , t\in\pi^n \right)$. Since the space $\p^n_D$ is isomorphic to $(0,+\infty)^n$, and since a coordinate system on it is given by $n$ length ratios, this amounts to proving that the $n$-tuples of length ratios of the polygonal curves $P_D(t)$ take all the values in $(0,+\infty)^n$. We thus have to express these length ratios, and to study their behavior when $t$ is varying in $\pi^n$, by taking into account the description by the Schlesinger system. We will obtain a more complicated expression than in Euclidean case, because of the existence of singularities on maxfaces. We will see how one can boiled down to the Euclidean case by using the implicit function theorem, and we will then briefly present the strategy of the proof developed in~\cite{Desideri}.

\bigskip

On the boundary, the singularities of a maxface $X_D(t)$ in $\X^n_D$ are isolated. This can be deduced from the reality properties of the Fuchsian system $(A_D(t))$ and of its fundamental solution $\Y_0(x,t)$, and also because along an edge $(a_i,a_{i+1})$, the Gauss map $N$ of $X_D(t)$ lies in the intersection of the timelike plane of normal vector $u_i$ with the sphere $\H^2$. Its stereographic projection $g$ thus takes values into a line that intersects the circle $|g|=1$ of singular values only twice. Since the function $g$ is meromorphic, it thus does not accumulate along the edge at a value $|g|=1$. Moreover, since $N$ is well-defined and timelike at the vertices $x=t_i$, the maxface $X_D(t)$ only has a finite number of boundary singularities. Let us remind that the induced metric of the maxface $X_D(t)$ is given by
\[
\dd s^2=\left(|G(x,t)|^2-|H(x,t)|^2 \right)^2|\dd x|^2
\]
where $\gxt$ is the first line of the fundamental solution $\Y_0(x,t)$. Since the sign of the quantity $\left|G(x,t)|-|H(x,t)\right|$ may change at those singular points (and we will see that it does change), we get the following expression of the length ratios of the polygonal boundary curve $P_D(t)$
\[
 r_i(t)  = \frac{\displaystyle\int_{t_i}^{t_{i+1}} \left||G(x,t)|^2 - |H(x,t)|^2\right| \dd x}{\displaystyle\int_0^1 \left||G(x,t)|^2 - |H(x,t)|^2 \right| \dd x},
\]
($i=1,\ldots,n$). Even if the fundamental solution $\Y_0(x,t)$ is defined up to multiplication by real scalars, for all $t\in\pi^n$, the length ratios $r_i(t)$ are correctly defined. We then define the length-ratio function $F_D$ associated with the oriented direction $D\in\D^n$
\[
 F_D:\pi^n \to (0,+\infty)^n, \qquad F_D(t)=(r_1(t),\ldots,r_n(t)),
\]
and we have to establish the following result.

\begin{thm}
For any given oriented direction $D\in\D^n$, the length-ratio function $F_D:\pi^n \to (0,+\infty)^n$ is surjective.
\label{thm-surj}
\end{thm}

Theorem~\ref{thm-surj} ends the proof of Theorem~\ref{thm-plateau-max}. This should perhaps be clarified for closed polygonal curves. Consider a value $r=(r_1,\ldots,r_n)$ in $(0,+\infty)^n$ such that the polygonal curve $P\in\p^n_D$ of length ratios $r$ is closed, in the sense that the two half-lines derived from $a_1$ and $a_{n+2}$, and of respective oriented directions $-D_{n+3}$ and $D_{n+2}$, intersect each other. Thanks to Theorem~\ref{thm-surj}, we know that there exists $t^0\in\pi^n$ such that the maxface $X_D(t^0)\in\X^n_D$ is bounded by the curve $P$. The question is whether the image of the point $t_{n+3}=\infty$ by $X_D(t^0)$ is still at infinity, or at the intersection of the two half-lines, \ie whether the $(n+2)$-th and $(n+3)$-th edge lengths of $X_D(t^0)$ are infinite, or not. If the last vertex $a_{n+3}=X_D(t^0)(\infty)$ is at infinity, the maxface has still an end. Since we prescribe the behavior of the maxfaces of $\X^n_D$ at their end, we know that $X_D(t^0)$ should then be asymptotic to a helicoid containing the two half-lines. When these intersect, the helicoid is not defined anymore, and the maxface can not have an end at $t_{n+3}=\infty$.

\bigskip

To prove Theorem~\ref{thm-surj}, we mainly have to simplify the expression of the function $F_D$ in Minkowski space, \ie in fact to eliminate the moduli, in order to prove that its behavior is the same as its analogue in Euclidean space. The proof of Theorem~\ref{thm-surj} will then be exactly the same. This is the main goal of the proof of the following proposition, which corresponds to Proposition 4.4 in~\cite{Desideri}.

\begin{prop}
For any given oriented direction $D\in\D^n$, the length-ratio function $F_D$ holomorphically extends onto a simply connected open neighborhood  $U'\subset\B^n$ of the simplex $\pi^n$.
\label{prop-F-holo-dans-pi}
\end{prop}

\begin{proof}
As in the Euclidean case, for every $i=1,\ldots,n+1$, we consider the fundamental solution $\Y_i(x,t):=\Y_0(x,t)\cdot S_i$ of the Fuchsian system $(A_D(t))$, where the matrix $S_i\in\SU(1,1)$ has been defined in the proof of Lemma~\ref{lemma-p-q-real}. Its first line $\left(g_i(x,t),h_i(x,t)\right)$ is real on the interval $(t_i,t_{i+1})$, and since $S_i\in\SU(1,1)$, we have
\[
 r_i(t)  = \frac{\displaystyle \int_{t_i}^{t_{i+1}} \left|g_i(x,t)^2-h_i(x,t)^2\right| \dd x}
 {\displaystyle \int_0^1  \left| g_{n+1}(x,t)^2-h_{n+1}(x,t)^2 \right| \dd x}.
\]
For every $t\in U$, we define the function $f_i(\cdot,t)$ by
\[
 f_i(x,t):=g_i(x,t)^2-h_i(x,t)^2,
\]
it is defined and holomorphic on the universal covering space of $\P\ssm S(t)$. The function $f_i$ is real valued when $t\in\pi^n$ and $t_i<x<t_{i+1}$. We set
\[
 \ell_i(t) := \int_{t_i}^{t_{i+1}} \left|f_i(x,t)\right|\dd x,
\]
and we then have $ r_i(t)  = \ell_i(t) / \ell_{n+1}(t)$. We intend to express the functions $\ell_i(t)$ without modulus, in order to be able to holomorphically extend them onto a neighborhood of the simplex $\pi^n$ in $\B^n$.

Let us suppose first that the parameter $t$ is fixed in $\pi^n$. On the real interval $(t_i,t_{i+1})$, the zeros of the function $f_i(\cdot,t)$, which are isolated, are exactly the singularities of the maxface $X_D(t)$. Since the Gauss map $N(x,t)$ of the maxface $X_D(t)$ is well-defined and timelike at the singular points $x=t_j$, then there exist open neighborhoods (depending on $t$) of the points $x=t_j$ on which the induced metric of the maxface $X_D(t)$ does not degenerate. The function $f_i(\cdot,t)$ thus only has a finite number of real zeros between $t_i$ and $t_{i+1}$, and this finite number $m_i$ of zeros (counted with multiplicity) is independent of $t$.

Let us show that the zeros of the function $f_i(\cdot,t)$ are simple. Consider such a point $z^0\in\P\ssm S(t)$ at which we have $f_i\left(z^0,t\right)=0$, \ie
\begin{equation}
 g_i(z^0,t)=\e h_i(z^0,t), \qquad \text{where } \e=\pm1.
\label{(gi-hi)}
\end{equation}
The functions $g_i(\cdot,t)$ and $h_i(\cdot,t)$ have no common zero, because it would be an apparent singularity of the Fuchsian system $(A_D(t))$, which does not have any. The value $h_i\left(z^0,t\right)$ is thus not null. The derivative of the function $f_i(\cdot,t)$
\[
 \frac{\partial f_i}{\partial x} = 2h_i\frac{\partial h_i}{\partial x} - 2g_i\frac{\partial g_i}{\partial x},
\]
thus vanishes at the point $x=z^0$ if and only if we have
\begin{equation}
 \frac{\partial g_i}{\partial x}\left(z^0,t\right) =\e \frac{\partial h_i}{\partial x}\left(z^0,t\right).
\label{(gi-hi-derivee)}
\end{equation}
The functions $g_i(\cdot,t)$ and $\e h_i(\cdot,t)$ are solutions of the same linear homogenous second-order differential equation (the Fuchsian equation $(E_D(t))$ associated with $(A_D(t))$). If both equalities~\eqref{(gi-hi)} and~\eqref{(gi-hi-derivee)} hold, then these two solutions would coincide everywhere, which is impossible. The point $z^0$ is thus a simple zero of the function $f_i(\cdot,t)$.

Let us denote by $z_j^i(t)$, $j=1,\ldots,m_i$, the real zeros of $f_i(\cdot,t)$ between $t_i$ and $t_{i+1}$
\[
 t_i<z_1^i(t)<\cdots<z_{m_i}^i(t)<t_{i+1},
\]
and set
\[
 z_0^i(t) = t_i ,\qquad z_{m_i+1}^i(t) = t_{i+1}.
\]
Let $\e_j^i\in\{-1,+1\}$ be the sign of the function $f_i(\cdot,t)$ on the interval $\left(z_j^i(t),z_{j+1}^i(t)\right)$. In particular: $\e_j^i \cdot \e_{j+1}^i=-1$. We have
\[
 \ell_i(t) = \sum_{j=0}^{m_i} \e_j^i \int_{z_j^i(t)}^{z_{j+1}^i(t)} f_i(x,t)\dd x.
\]
Consider now small complex variations of the parameter $t$. To prove that the functions $\ell_i(t)$ are holomorphic, we have to prove in particular that the zeros $z_j^i(t)$, which are well-defined for $t\in\pi^n$, are real analytic. This is given by the implicit function theorem, since for $t\in\pi^n$ the $z_j^i(t)$ are defined by
\[
 f_i\left(z_j^i(t), t\right) =0,
\]
and since we have proved
\[
 \frac{\partial f_i}{\partial x}\left(z_j^i(t), t\right) \neq 0.
\]
For each value $t^0\in\pi^n$ and each $i,j$, we thus obtain a simply connected neighborhood $U^i_j$ of $t^0$ in $\B^n$ on which the function $z_j^i:U_j^i \to \C$ is holomorphic and verifies
\[
 \forall t\in U_j^i \quad f_i\left(z_j^i(t),t\right)=0.
\]
Taking the intersection of the open sets $U_j^i$, we obtain a simply connected neighborhood of $t^0$ in $\B^n$ on which all the functions $z_j^i$ are holomorphic.

The end of the proof is then the same as in the Euclidean case (see Proposition~4.4 in~\cite{Desideri}): the holomorphicity of the functions $f_i(x,t)$ with respect to $t$ is deduced from Proposition~\ref{prop-Ai-holo}, and we then conclude by applying the dominated convergence theorem.
\end{proof}

Since the regularity properties of the functions $g_i(x,t)$ and $h_i(x,t)$ are passed on to the functions $f_i(x,t)$ and also, when $x$ is fixed, to the $z_j^i(t)$, the length-ratio function $F_D$ exactly behaves as its Euclidean analogue, even on the boundary of the simplex $\pi^n$. The rest of the proof of Theorem~\ref{thm-surj}  is then the same as the one of Theorem 4.1 in~\cite{Desideri}. We give a sketch of the proof, which mainly relies on the behavior of the function $F_D$ at the boundary of the simplex $\pi^n$. The boundary of $\pi^n$ is formed of faces that are lower-dimensional simplexes, characterized by equalities of the type $t_i=t_{i+1}$, that is to say by the fact that some singularities $t_i$ are ``missing'', because they coincide with the following singularity. Let us consider a face $P^k$ of dimension $k$ of the boundary of $\pi^n$ ($0\leq k\leq n-1$). It is homeomorphic to the simplex $\pi^k$, and it is characterized by $n-k$ ``missing'' singularities $t_i$. We define the $(k+3)$-tuple of oriented directions $D'\in\D^k$, obtained from $D= \left (D_1,\ldots,D_{n+3} \right)$ by ``removing'' the directions $D_i$ corresponding to the singularities $t_i$ that are missing. We established in~\cite{Desideri} the following proposition.

\begin{prop}
 The function $F_D$ continuously extends on the face $P^k$. Moreover, its restriction to $P^k$ coincides, up to homeomorphisms, with the lower-dimensional length-ratio function
\[
 F_{D'}:\pi^k \to (0,+\infty)^k.
\]
\label{prop-fundamental-max}
\end{prop}

Even if the geometrical meaning of this result is natural, its proof constitutes the most difficult step of the proof of Theorem~\ref{thm-surj}. It is based on the behavior of the solutions of the Schlesinger system at its fixed singularities, that is to say at the points $t$ such that $t_i=t_j$, $i\neq j$. This is a known part of Garnier's work~\cite{Garnier26}, that has been further developed and generalized by Sato, Miwa et Jimbo~\cite{SMJ}. By fitting these results to our situation, we established Proposition~\ref{prop-fundamental-max}.

An induction on the number $n+3$ of vertices then enables us to conclude the proof of Theorem~\ref{thm-surj}. By identifying the simplexes $\pi^n$ and $(0,+\infty)^n$, we obtain a function $\w F_D:(0,+\infty)^n\to(0,+\infty)^n$ which is surjective if and only if $F_D$ is. The induction hypothesis is that for any $k=1,\ldots,n$, and for any oriented direction $D\in\D^k$, the function $\w F_D$ is of degree $1$, that is to say, is homotopic to the identity on $(0,+\infty)^k$. The basis of the induction for $n=1$ is a direct consequence of Proposition~\ref{prop-fundamental-max}. For the inductive step, we need the following topological result as well, which is proved in~\cite{Desideri}, Proposition~4.5.

\begin{prop}
Let $K\subset \R^n$ be a convex compact, and $f:K\to K$ a continuous function on $K$. If $f(\partial K) \subset \partial K$, and if the function $f\mid_{\partial K}: \partial K\to \partial K$ is of degree $1$, then the function $f:K\to K$ is of degree $1$.
\label{prop-degree}
\end{prop}

This ends the proof of Theorem~\ref{thm-surj}, and thus the one of Theorem~\ref{thm-plateau-max} as well.

\bibliographystyle{plain}
\bibliography{biblio}

\end{document}